\documentclass[12pt,leqno,fleqn]{amsart}  
\usepackage{amsmath,amstext,amsthm,amssymb,amsxtra}

\usepackage{tikz}
\usepackage{txfonts} % pxfonts txfonts 
\usepackage[T1]{fontenc}
\usepackage{lmodern}

 \usepackage{euler}   % better than the option below

\usepackage[msc-links]{amsrefs}

\usepackage{mathtools}
\mathtoolsset{showonlyrefs,showmanualtags} 

\usepackage{hyperref} %,hypertexnames=false,colorlinks,[pagebackref]
\hypersetup{
    colorlinks=true,       % false: boxed links; true: colored links
    linkcolor=blue,          % color of internal links
    citecolor=magenta,        % color of links to bibliography
    filecolor=magenta,      % color of file links
    urlcolor=cyan           % color of external links
}

\allowdisplaybreaks
\swapnumbers

\theoremstyle{plain} % definition 
\newtheorem{lemma}[equation]{Lemma} 
 
\newtheorem{theorem}[equation]{Theorem} 
 
\newtheorem{conjecture}[equation]{Conjecture}
\newtheorem{question}[equation]{Question}

\theoremstyle{definition}
\newtheorem{definition}[equation]{Definition} 
\newtheorem{notation}[equation]{Notation}

\theoremstyle{remark}
\newtheorem{remark}[equation]{Remark}

\numberwithin{equation}{section}

\title[Hilbert Transform, Coronas and Energy]{The Two Weight Inequality for Hilbert Transform, Coronas, and Energy Conditions}
\author[M.T. Lacey]{Michael T. Lacey}
\address{School of Mathematics \\
Georgia Institute of Technology \\
Atlanta GA 30332 }
\email{lacey@math.gatech.edu}
\thanks{Research supported in part by grant NSF-DMS 0968499.}
\author[E.T. Sawyer]{Eric T. Sawyer}
\address{ Department of Mathematics \& Statistics, McMaster University, 1280
Main Street West, Hamilton, Ontario, Canada L8S 4K1 }
\email{sawyer@mcmaster.ca}
\thanks{Research supported in part by NSERC}
\author[C.-Y. Shen]{Chun-Yen Shen}
\address{Department of Mathematics \& Statistics, McMaster University, 1280
Main Street West, Hamilton, Ontario, Canada L8S 4K1 }
\email{cyshen@mcmaster.ca}
\author[I. Uriarte-Tuero]{Ignacio Uriarte-Tuero}
\address{ Department of Mathematics \\
Michigan State University \\
East Lansing MI }
\email{ignacio@math.msu.edu}
\thanks{Research supported in part by the NSF, through grant DMS-0901524.}
\date{}

\begin{document}

\begin{abstract}
We consider the two weight problem for the Hilbert transform, namely the question of finding real-variable 
characterization of those pair of weights for which the Hilbert transform acts boundedly on $ L ^2 $ of the weights. 
Such a characterization is known subject to certain side conditions.  We give a new proof, simpler in many details,
of the best such result. 
In addition, we analyze underlying assumptions in  the proof,  especially in terms of   two alternate side conditions. 
A new characterization in the case of one doubling weight is given.  
\end{abstract}

\maketitle
\setcounter{tocdepth}{1}
\tableofcontents 

%%%%%%%%%%%%%%%%%%%%%%%%%%%%%% SECTION  SECTION SECTION
%%%%%%%%%%%%%%%%%%%%%%%%%%%%%% SECTION  SECTION SECTION 
\section{Introduction} %\label{s.}

A \emph{weight} is a non-negative Borel measure.
We are interested in the two weight question for the Hilbert transform:   
For two weights $ (\sigma , w )$,  characterize the $ L ^2 $ inequality 
\begin{equation}\label{e.H2}
\lVert H (\sigma f )\rVert_{ L ^2 (w )} \le \mathbf B \lVert f \rVert_{ L ^2 (\sigma )} \,. 
\end{equation}
Here, the inequality is understood in the sense that there is a uniform bound on the operator norm of a standard truncation 
on the singular integral kernel.  Throughout, we will write $ H _{\sigma } f = H (\sigma f)$, and understand at all times that  some truncation is in place.  The inequality  above is in its self-dual formulation: Interchange the roles of $ w$ and $ \sigma $ to get the dual inequality. 
We are also focused on $ L ^2 $ inequalities, so throughout we use the abbreviation $ \lVert f\rVert_{\sigma } := \lVert f\rVert_{ L ^2 (\sigma )}$. This conjecture, due to Nazarov-Treil-Volberg \cite{V}, has been the focus of attention. 

%%%%%%%%%%%%%%%%%%%%%%%%%%%%%% CONJECTURE CONJECTURE CONJECTURE
\begin{conjecture}\label{j.2wtH}  For a pair of weights $ (w ,\sigma )$ we have the 
	inequality \eqref{e.H2} if and only if these three  constants are finite. 
\begin{gather}\label{e.A2}
  \mathbf A_2:= \sup _{x \in \mathbb R } \sup _{t>0}  \mathsf P w  (x,t) \mathsf P \sigma  (x,t) 
\,, 
\\ \label{e.H1}
\mathbf H ^2 := \sup _{I} \sigma (I) ^{-1} 
\int _{I} \lvert H (\sigma \mathbf 1_{I})\rvert ^2 \; w (dx) \,, 
\\ \label{e.H*1} 
 \mathbf H _{\ast } ^2 := 
\sup _{I} \sigma (I) ^{-1} 
\int _{I} \lvert H (w  \mathbf 1_{I})\rvert ^2 \; \sigma  (dx)   \,, 
\end{gather}
where in the first line,  $  \mathsf P w  (x,t)$ denotes the Poisson extension of $ w $ 
to the upper half plane.  In particular, the first line is an extension of the classical $ A_2$ condition, 
and  is referred to herein as the $ A_2$ condition.  
The next  two conditions are  dual to one another,  and  are 
referred to as the \emph{testing conditions.}
\end{conjecture}
%%%%%%%%%%%%%%%%%%%%%%%%%%%%%% CONJECTURE CONJECTURE CONJECTURE

We will keep track of certain constants, like the three constants in the Conjecture above. Many of these will come in 
dual pairs, namely with the roles of $ w$ and $ \sigma $ reversed. An  asterisk subscript will denote the dual constant, obtained through exchanging the roles 
of the two weights.  
The exact form of the Poisson integral is not important for us, and throughout we will use this form of it. 
For weight $ \sigma $ and interval $ I$, we set 
\begin{equation}\label{e.Pdef}
\mathsf P (\sigma , I) := \int \frac { \lvert  I\rvert } { (\lvert  I\rvert + \textup{dist} (x,I) ) ^2 } \sigma (dx) \,. 
\end{equation}
This is the same, up to constants, as evaluating the usual Poisson extension of $ \sigma $ at the  $ (c, \lvert  I\rvert )$, 
where $ c$ is any point of $ I$.

To date, the Conjecture above  has only been verified for pairs of weights which satisfy side conditions, which help control certain degeneracies in the weights  $ \sigma  $ and $ w $.  
These side conditions are inspired by the \emph{Pivotal Conditions} of \cite{1003.1596}, 
and were expanded and refined in \cite{1001.4043}, using the notion of \emph{energy}.  
Our purposes are two-fold. (1) We will give a notably simpler proof of the best known current estimates.  
(2) We will analyze the proof strategy, introducing new side conditions sufficient for the two-weight estimate. 
These new side conditions are themselves, in a sense to be made precise in \S\ref{s.splitting}, a consequence of the correctness 
of the proof strategy.  (3) We point out in  Question~\ref{q1}, that it is not known if the proof strategy applies to all pairs of weights for  which satisfy the two weight inequality.   
 A new characterization when just one weight is doubling will follow from this analysis.

We define for interval $ I$ the \emph{energy of $ w$ over $ I$} to be 
\begin{equation}\label{e.EnergyDef}
\mathsf E (w , I) ^2 := 
w (I) ^{-2} \int _{I} \int _{I} 
\frac{ \lvert  x-x'\rvert ^2  } {\lvert  I\rvert ^2  } w (dx) w (dx') \,. 
\end{equation}
Assuming that $ \lvert  I\rvert=1 $, and $ w (I)=1$, this is two times the square of the distance, in the $ L ^2 (w \mathbf 1_{I})$ metric,  of the function  $ x \mathbf 1_{I}$ to the linear space of constants.  
The \emph{energy constant} of a pair of weights $ (\sigma ,w )$ is the smallest constant $ \mathbf E$ for which 
the following inequality holds. For all intervals $ I_0$ and all partitions $ \{I_j \;:\; j\ge 1\}$ of $ I_0$ we have 
\begin{equation}\label{e.EnergyConstant} 
\sum_{j\ge 1} \mathsf P (\sigma \cdot I_0, I_j) ^2 \mathsf E (w , I_j) ^2 w (I_j) \le \mathbf E ^2 \sigma (I_0) \,. 
\end{equation}
Here, inside the Poisson integral, we are identifying the interval $ I_0$ with its indicator function, which we will do throughout,  as this will be a  convenience in the heart of the proof.  

A fundamental observation is that the energy constant is finite if the  $ A_2$ constant and the testing conditions 
\eqref{e.H1} and \eqref{e.H*1} hold.  Namely, it was proved in \cite{1001.4043} that we have $ \mathbf E \lesssim \mathbf A_2 + \mathbf 
H$.  This depends upon the specific character of the $ 1/y$ kernel; its modification for other relevant singular integrals is not 
nearly as simple. 

We turn to the side conditions we need for our Theorem.  Fix a choice of $ 0< \epsilon < \frac 12$, and integer $ r\ge 2$. 
We say that  a pair of intervals $ (I,J)$ are \emph{$(\epsilon,r)$-good} 	
if for all  $ J \subset I$, satisfying $ \lvert  J\rvert\le 2 ^{-r} \lvert  I\rvert  $,  it follows that 
$ \textup{dist} (J, \partial I) \ge \lvert  I\rvert ^{1- \epsilon } \lvert  J\rvert ^{\epsilon }  $.

%%%%%%%%%%%%%%%%%%%%%%%%%%%%%%  DEFINITION DEFINITION DEFINITION
\begin{definition}\label{d.side} The  \emph{Dini energy constant} of   pair of weight $ (\sigma , w )$   is  the smallest constant $ \boldsymbol\Psi $ for which the following inequality holds: There is a decreasing non-negative sequence $ \psi (s)$ with $ \sum_{s\ge 1} \psi (s)=1 $, 
so that for all integers $ s$
\begin{equation}\label{e.PsiI}
\psi (s) ^{-2} 
\sum _{ j, k\ge 1}  \mathsf P (\sigma \cdot (I_0 - I_j ), I_{j,k}) ^2 \mathsf E (w , I_{j,k}) ^2 w (I_{j,k}) \le  {\boldsymbol\Psi }^2 \sigma (I_0) \,. 
\end{equation}
In this inequality, we have these conditions. 
%%  ENUMERATE
\begin{enumerate}
\item $ I_0$ is an interval and  $ \{I_j \;:\; j\ge 1\}$ a partition  of $ I_0$.
\item  We have 
	secondary  partitions of $ I_j$ into intervals $ \{I _{j,k} \;:\; k\ge 1\}$, so that the  pair of  intervals  $ (I_j, I _{j,k})$ are $(\epsilon,r)$-good for all $j, k\ge 1$. 
\item  We have     $ \lvert  I _{j,k}\rvert < 2 ^{- s} \lvert  I_j\rvert  $ for all $ j,k\ge 1$.   
\end{enumerate}
%% ENUMERATE
Note that here, it is certainly required that we consider the Poisson integral  of $ \sigma $ restricted to the complement of $ I_j$, else we could not expect to get the required  decay in $s$ to make the supremum finite.    
\end{definition}
%%%%%%%%%%%%%%%%%%%%%%%%%%%%%%  DEFINITION DEFINITION DEFINITION

This is very close to the side condition considered in \cite{1001.4043}, and is weaker than the pivotal condition of \cite{1003.1596}.   
Namely, there is a pair of weights which \emph{fail} one direction of the  Pivotal Condition, but 
satisfy both directions of  the side condition above, for $ \psi (s) \simeq 2 ^{- \epsilon s/2}$, and the Hilbert transform is bounded for this 
	pair of weights. 

%%%%%%%%%%%%%%%%%%%%%%%%%%%%%% THEOREM THEOREM THEOREM
\begin{theorem}\label{t.main} Let $ w ,\sigma $ be two weights which do not 
	share any common point mass, and for some  $ 0< \epsilon <1$ and integer $ r$, 
	have finite  Dini Energy Constant $ \boldsymbol\Psi $, and finite 
dual  Dini Energy Constant $ \boldsymbol\Psi _{\ast }$. Then  Conjecture~\ref{j.2wtH}  holds.  Namely, 
we have the two weight inequality \eqref{e.H2} 
if and only if   the $ A_2$ condition \eqref{e.A2},   and the two testing conditions in 
\eqref{e.H1} and \eqref{e.H*1} hold.
\end{theorem}
%%%%%%%%%%%%%%%%%%%%%%%%%%%%%% THEOREM THEOREM THEOREM

This theorem is essentially contained in \cite{1001.4043}, but the current proof contains  many simplifications. 
Basic to the proofs are corona decompositions. We introduce herein a Calder\'on-Zygmund corona, whose use 
precludes the need for nuanced Carleson measure estimates.   We still need a sophisticated corona decomposition 
modeled on one in \cite{1003.1596}, but in the current formulation we can again avoid appeals to Carleson measure 
estimates.  Prior arguments required a  number of such arguments.

\medskip 

One of us initiated the study of two weight inequalities for the maximal function \cite{MR676801} and fractional integrals \cite{MR930072}.
Cotlar and Sadosky have established two weight variants of the Helson--Szeg\H o theorem \cite{MR730075}, providing a complex analytic 
solution to the two weight problem. 
The dyadic variant of the Nazarov-Treil-Volberg  conjecture is proved in \cite{MR2407233}.  Two weight inequalities for maximal truncations of singular integrals  are studied  by a completely different method in \cite{0805.0246}.  This paper represents, in some sense, a unification of these two 
lines of approach.  
The two weight problem for the Hilbert transform is 
closely related to a number of subjects, including  embedding inequalities for model space \cite{NV} and de Branges space \cite{BMS}; 
interpolating sequences for Paley-Weiner space \cite{MR1617649}; and  spectral theory for perturbed operators \cite{MR2543858}.

%%%%%%%%%%%%%%%%%%%%%%%%%%%%%% SECTION  SECTION SECTION
%%%%%%%%%%%%%%%%%%%%%%%%%%%%%% SECTION  SECTION SECTION 
\section{Dyadic Grids and Haar Functions} %\label{s.}

%%%%%%%%%%%%%%%%%%%%%%%%%%%%%% SUBSECTION* SUBSECTION* SUBSECTION* SUBSECTION*
 %%%%%%%%%%%%%%%%%%%%%%%%%%%%%% SUBSECTION* SUBSECTION* SUBSECTION* SUBSECTION* 
\subsection*{Dyadic Grids.}%\label{ss.}
A collection of intervals $ \mathcal G$ is a \emph{grid} if for all $ G,G'\in \mathcal G$, we have $ G\cap G' \in \{\emptyset , G, G'\}$.  
By a \emph{dyadic grid} we mean a grid  $ \mathcal D $ of intervals of $ \mathbb R $ such that 
for each interval $ I\in \mathcal D$, the subcollection 
$ \{ I' \in \mathcal D \;:\; \lvert  I'\rvert= \lvert  I\rvert  \}$ partitions $ \mathbb R $, aside from endpoints 
of the intervals.  In addition, the  left and right halves of $ I$, denoted by $ I _{\pm}$, are also in $ \mathcal D$.  

For $I\in \mathcal{D}$, the left and right halves  $I_{\pm}$ 
are referred to as the \emph{children} of $I$.  We denote by $\pi _{\mathcal{D}}\left( I\right) $ the unique interval
in $\mathcal{D}$ having $I$ as a child, and we refer to $\pi _{\mathcal{D}%
}\left( I\right) $ as the $\mathcal{D}$-parent of $I$.

There is no unique choice of $\mathcal D $. 
To accomodate the notion of an interval being $ (\epsilon ,r)$-good, 
one must  make a \emph{random} selection of grids, but 
we have nothing to contribute to this portion of the proof. We refer the reader to \cites{1001.4043,1003.1596} for more 
details on this point. 

%%%%%%%%%%%%%%%%%%%%%%%%%%%%%% SUBSECTION* SUBSECTION* SUBSECTION* SUBSECTION*
 %%%%%%%%%%%%%%%%%%%%%%%%%%%%%% SUBSECTION* SUBSECTION* SUBSECTION* SUBSECTION* 
\subsection*{Haar Functions.}%\label{ss.}
Let $ \sigma $ be a weight on $ \mathbb R $, one that does not assign positive mass to any endpoint 
of a dyadic grid $ \mathcal D$. We define the Haar functions associated to $ \sigma $ as follows. 
\begin{align}\label{e.hs1}
	h_{I}^{\sigma} & := \sqrt{\frac{\sigma(I_{-}) \sigma( I_{+})} {\sigma( I)}} 
\Biggl(  \frac{{I_{-}}}{\sigma( I_{-})}-\frac{{I_{+}}}{\sigma( I_{+})} \Biggr)\,.    
\end{align}
In this definition, we are identifying an interval with its indicator function, and we will do so  throughout the remainder of the paper.  
This is  an $ L ^2 (\sigma )$-normalized  function, and   has $ \sigma $-integral zero.
For any dyadic interval $ I ^{0}$, it holds that 
$  \{\sigma (I_0) ^{-1/2} {I_0}\} \cup \{ h ^{\sigma} _I \;:\; I\in \mathcal D\,, I \subset I_0\}$ is an orthogonal basis for $ L ^2 ( I_0\sigma )$.  

We will use the notation 
\begin{align}\label{e.mart}
\Delta ^{\sigma} _{I}f &= \langle f, h ^{\sigma} _{I} \rangle _{\sigma} h ^{\sigma} _{I}
= {I _{+}} \mathbb E ^{\sigma} _{I _{+}} f + 
{I _{-}} \mathbb E ^{\sigma} _{I _{-}}f - {I} 
\mathbb E ^{\sigma} _{I} f \,.  
\end{align}

The second equality is the familiar martingale difference equality, and so 
we will refer to $ \Delta ^{\sigma} _{I} f$ as a martingale difference.  It implies the familiar telescoping identity 
$
	\mathbb E _{J} ^{\sigma }f = \sum_{ I \;:\; I\supsetneq J} \mathbb E _{J} ^{\sigma } \Delta ^{\sigma } _{I} f \,. 
$
Finally, we will need the estimate below, which follows immediately from Cauchy-Schwartz. 
\begin{equation}  \label{e.h<}
\bigl\lvert  \mathbb E ^{\sigma } _{I _{\pm}  }h ^{\sigma } _{I }\bigr\rvert 
\le  \sigma (I _{\pm}) ^{-1/2} \,. 
\end{equation}

%%%%%%%%%%%%%%%%%%%%%%%%%%%%%% SECTION  SECTION SECTION
%%%%%%%%%%%%%%%%%%%%%%%%%%%%%% SECTION  SECTION SECTION 
\subsection*{Good-Bad Decomposition} %\label{s.} 

With a choice of dyadic grid $ \mathcal D$ understood, we then slightly change the definition of $ (\epsilon ,r)$-good.  
We say that $ J\in \mathcal D$ is \emph{$ (\epsilon ,r)$-good} if and only if  for all intervals $ I \in \mathcal D$ with $ \lvert  I\rvert\ge 2 ^{r+1}  \lvert  J\rvert $, we have that the distance from $ J$ to the boundary of \emph{either child} of $ I$ is at least $ \lvert  J\rvert ^{\epsilon } \lvert  I\rvert ^{1- \epsilon }$.    

For $ f\in L ^2 (\sigma )$ we set $ P _{\textup{good}} ^{\sigma } f = \sum_{ \substack{I \in \mathcal D \\  \textup{$ I$ is $(\epsilon,r)$-good}} }  \Delta ^{\sigma } _{I  }f $.  The projection $ P _{\textup{good}} ^{w } \phi  $ is defined similarly. 
Important elements of the suppressed construction of random grids \cite{1001.4043,1003.1596} are  that 
%%  ENUMERATE
\begin{enumerate}
\item It suffices to consider a single dyadic grid $ \mathcal D$, but 
	we will sometimes write $ \mathcal D ^{\sigma } $ and $ \mathcal D ^{w}$ to emphasize the role of the two weights. 

\item For any fixed $ 0<\epsilon< \frac 12 $, we can choose integer $ r$ sufficiently large so that  it suffices 
	to consider $ f$ such that $ f= P _{\textup{good}} ^{\sigma } f$, and likewise for $ \phi \in L ^2 (w)$.
\end{enumerate}
%% ENUMERATE
Concerning the last property, this is, at some moments, an essential property. We suppress it in notation, however 
taking care to emphasize in the text those places in which we appeal to the property of being good.  

%%%%%%%%%%%%%%%%%%%%%%%%%%%%%% SECTION  SECTION SECTION
%%%%%%%%%%%%%%%%%%%%%%%%%%%%%% SECTION  SECTION SECTION 
\section{Analysis of the Splitting Assumption} \label{s.splitting}

Our principal concern is the bilinear form $ B (f,\phi ) := \langle H _{\sigma } f, \phi  \rangle _{w}$; let  $ \mathbf B$ be the 
best constant in the inequality $ \lvert  B (f, \phi )\rvert \le \mathbf B \lVert f\rVert_{\sigma }  \lVert \phi \rVert_{w} $. 

We define two more forms here.  
Throughout the paper by $ J \Subset I $ we mean that $ I,J$ are dyadic intervals, in a fixed dyadic grid, and $ J\subset I$ with 
$ \lvert  J\rvert\le 2 ^{-r} \lvert  I\rvert  $, with $ r$ the fixed integer in the $ (\epsilon,r)$-good property. Define 
\begin{equation*}
B _{\Subset } (f,\phi ) 
:= \sum_{I  \in \mathcal D ^{\sigma }} \sum_{J \in \mathcal D ^{w} \;:\; J\Subset I} 
 \mathbb E _{I_J} ^{\sigma } \Delta ^{\sigma } f \cdot \langle H _{\sigma } I_J , \Delta ^{\sigma } _{J} \phi  \rangle _{w} 
\end{equation*}
where $ I_J$ denotes the child of $ I$ that contains $ J$.  And, as mentioned in the previous section, we will identify 
an interval and its indicator function.  Denote by $ B _{\Supset} (f, \phi )$ the \emph{dual} bilinear form obtained by 
interchanging $ w$ and $ \sigma $.   See Figure~\ref{f.1} for a diagram illustrating the definition of these two forms.  
 Set $ \mathbf B_{\Subset}$ be the best constant in the inequality 
\begin{equation*}
\lvert B_{\Subset} (f, \phi ) \rvert \le \mathbf B_{\Subset} \lVert f\rVert_{\sigma } \lVert \phi \rVert_{w} \,, 
\end{equation*}
and $ \mathbf B_{\Supset} $ be the best constant in inequality for the dual bilinear form.

%%%%%%%%%%%%%%% Figure
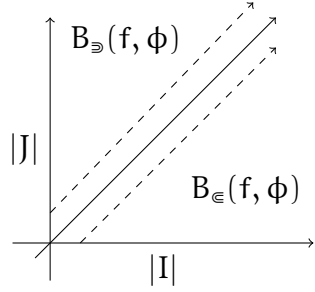
\begin{figure}
\begin{tikzpicture}
	\draw[->] (-.5,0) -- (3.5,0) node[midway,below] {$ \lvert  I\rvert $};
	\draw[->] (0,-.5) -- (0,3) node[midway,left] {$ \lvert  J\rvert $};
	\draw[->] (-.2,-.2) -- (3,3); 
	\draw[->,dashed] (0,.4) -- (2.7,3.2);
	\draw[->,dashed] (.4,0) -- (3,2.6);
	\draw (1,2.7) node {$ B _{\Supset} (f, \phi )$}; 
	\draw (2.6,.7) node {$ B _{\Subset} (f, \phi )$}; 
\end{tikzpicture}
\caption{A schematic diagram for the two forms $ B _{\Subset}$ and $ B _{\Supset}$. The dashed lines around the 
diagonal indicate that the terms associated with $ 2 ^{-r} \lvert  J\rvert\le \lvert  I\rvert\le 2 ^{r} \lvert  I\rvert   $ are treated 
in Theorem~\ref{t.preSplit}.} 
\label{f.1}
\end{figure}
%%%%%%%%%%%%%%% Figure

In order to state our first main result, we need one more constant.  
  Let $ \mathbf W$ be the best constant in the inequality 
  \begin{equation} \label{e.W}
 	 \bigl\lvert \langle H _{\sigma } I, J \rangle _{w}\bigr\rvert\le \mathbf W \sigma (I) ^{1/2} w (J) ^{1/2} \,, 
\end{equation}
where $ I$ and $ J$ are intervals with $ 2 ^{-r} \lvert  J\rvert \le \lvert  I\rvert\le 2 ^{r} \lvert  J\rvert   $. 
Recall that the integer $ r$ is fixed. It is known that $ W \lesssim \mathbf A_2 + \min \{\mathbf H, \mathbf H _{\ast }\}$.

%%%%%%%%%%%%%%%%%%%%%%%%%%%%%% THEOREM THEOREM THEOREM
\begin{theorem}\label{t.preSplit} Assume that the pair of weights satisfy the $ A_2$ bound,  and the two interval testing conditions 
\eqref{e.H1} and \eqref{e.H*1}. 
Then, assuming $ f= P ^{\sigma } _{\textup{good}} f$ and likewise for $ \phi $, it holds that 
\begin{equation} \label{e.preSplit}
\bigl\lvert  B (f,\phi ) - \bigl\{B _{\Subset} (f, \phi ) + B _{\Supset} (f, \phi ) \bigr\} \bigr\rvert \lesssim
\{ \sqrt {\mathbf A_2}+   \mathbf W \} 
\lVert f\rVert_{\sigma } \lVert \phi \rVert_{w} \,. 
\end{equation}
\end{theorem}
%%%%%%%%%%%%%%%%%%%%%%%%%%%%%% THEOREM THEOREM THEOREM

That is, the boundedness of $ H$ is equivalent to that of the sum  $ B _{\Subset} + B _{\Supset}$.
The remainder of the sufficiency proof for the main theorem is based upon the assumption that 
$  B_{\Subset}(f,g)$ and $  B_{\Supset}(f,g)$ are \emph{bounded independently of each other.}  It is commonplace in classical settings 
that this assumption holds.\footnote{In various $ T1$ theorems, there are canonical choices of paraproducts, which are bounded 
by the assumptions of the $ T1$ theorem, whence they are freely added and subtracted in the proof.  In the current setting, there is 
no canonical choice of paraproducts.} This brings up the following 

\begin{question} \label{q1} Let $ w, \sigma $ be a pair of weights.
	  Does it hold that $\mathbf B _{\Supset} + \mathbf B _{\Subset} \lesssim \mathbf B $? 
\end{question}

  Without an answer to this question, we cannot be sure that the approach to the two weight question  used in this paper, and in \cite{1001.4043,1003.1596} can even succeed.  Currently, there is no other approach to this question.

We introduce two new side conditions, more general, and more complicated, than the Dini condition; these conditions are phrased in terms of a dyadic 
grid, which is after all not fixed. The purpose in phrasing them is to provide precise description of those 
objects which the side conditions control. 
\begin{notation}
	For $ \mathcal F $ a subset of the dyadic grid $ \mathcal D$, it is convenient to refer to $ \mathcal F$ as a sub-tree of the  dyadic grid, and it useful to think of moving up or down the $ \mathcal F$-tree, moving by inclusion.  
For dyadic $ I \in \mathcal D$, we set $ \pi _{\mathcal F} I$,  \emph{the $ \mathcal F$-parent of $ I$,}   to be the minimal  element of $ \mathcal F$ that contains $ I$.
We set $  \pi ^{1} _{\mathcal F} I= \pi _{\mathcal F} I$, and inductively define  $ \pi ^{t+1} _{\mathcal F} I$ 
to be the minimal element of $ \mathcal F$ that strictly contains $ \pi ^{t} _{\mathcal F} I$.  
This has the consequence that if $ F\in \mathcal F$, then $ \pi _{\mathcal F} ^{1}F=F$. 
We write $ \textup{Child} _{\mathcal F} (F)$ 
for the maximal elements of $ \mathcal F$ which are strictly contained in $ F$, and call them the \emph{$ \mathcal F$-children of $ F$.} 
\end{notation}

%%%%%%%%%%%%%%%%%%%%%%%%%%%%%%  DEFINITION DEFINITION DEFINITION  
\begin{definition}\label{d.fStopping} Given interval $ I^0$ we set $ {\mathcal F} (I_0)$ to be the maximal  dyadic subintervals $ F$ such that  $  \mathbb E _{F} ^{\sigma } \lvert  f \rvert > 4  \mathbb E ^{\sigma }  _{I_0}   \lvert f\rvert  $.    
We set $  {\mathcal F} _0 = \{I ^{0}\}$, and inductively set 
\begin{equation*}
{\mathcal F} _{j+1} := \bigcup _{F\in {\mathcal F} _{j}} {\mathcal F} (F)\,. 
\end{equation*}
Then, the collection of \emph{$ f$-stopping intervals} is 
$ {\mathcal F} := \bigcup _{j=0} ^{\infty} {\mathcal F} _{j}$. 
\end{definition}
%%%%%%%%%%%%%%%%%%%%%%%%%%%%%%  DEFINITION DEFINITION DEFINITION

A basic fact, a consequence of the universal maximal function estimate, is 
\begin{equation}\label{e.F<}
\sum_{F\in {\mathcal F}}  \gamma (F)^2 \sigma (F) \lesssim \lVert f\rVert_{\sigma } ^2\,, 
\qquad  \gamma (F) = \mathbb E ^{\sigma } _{F}  \lvert f\rvert\,. 
\end{equation}  
This is referred to as the quasi-orthogonality condition.

We take $ f$ non-negative and supported on an interval $ I ^{0}$, and $f$-stopping intervals as above. 
Let $ \{g _{F} \;:\; F\in \mathcal F\}$ be a collection of functions in $ L ^2 (w)$ so that  for each $ F$,
%%  ENUMERATE
\begin{enumerate}
\item   $ g_F$ is supported on $ F$  and constant on $ F'\in\textup{Child} _{\mathcal F} (F)$; 
	
\item letting $ \mathcal J ^{\ast} (F)$ be the maximal intervals $ J ^{\ast} $ such that  $ J ^{\ast} \Subset F$, $ J ^{\ast} $ is $ (\epsilon ,r)$-good,  and $ \pi _{\mathcal F} J ^{\ast} = F  $,  we have $ \mathbb E ^{w} _{J ^{\ast} } g_F=0$  for each $ J ^{\ast} \in \mathcal J ^{\ast} $.  
\end{enumerate}
%% ENUMERATE
We say that $ g_F$ is \emph{$ \mathcal F$-adapted to $ F$.} 
Let $ \mathbf F$ be the smallest constant in the inequality below, holding for all non-negative $ f \in L ^2 (\sigma )$,   and collections $\{g_F \}$ as just described.
\begin{equation}\label{e.funcEnergy}
	\sum_{F\in \mathcal F} \sum_{J ^{\ast} \in \mathcal J ^{\ast} (F)} 
	P (f ({\mathbb R -F}) \sigma , J ^{\ast} ) \bigl\langle  \frac {x} {\lvert  J ^{\ast} \rvert }, g_F {J ^{\ast} }\bigr\rangle _{w}
	\le \mathbf F 
	\lVert f\rVert_{\sigma } \Bigl[\sum_{F\in \mathcal F} \lVert g _F\rVert_{w} ^2  \Bigr] ^{1/2} \,. 
\end{equation}
We refer to this as the \emph{functional energy condition}.  Taking  $ \mathcal F$  to be a partition of an interval $ I_0$, 
and $ f = {I_0}$, we can recover the energy condition \eqref{e.EnergyConstant}. We denote by $ \mathbf F _{\ast }$ as the dual condition, with the roles of $ w$ and $ \sigma $ reversed. 

\medskip 

The second condition is as follows.  We write $ f\in BF _{\mathcal F} (F)$, and say that $ f$ is of \emph{bounded fluctuation } if 
(i) $ f$ is supported on $ F$, 
(ii) $ f$ is constant on each $ F'\in \textup{Child} _{\mathcal F} (F)$, 
and (ii) for each dyadic interval $ I \subset F$, which is \emph{not contained in } some $ F'\in \textup{Child} _{\mathcal F} (F)$, we have $ \mathbb E ^{\sigma } _{I} \lvert  f\rvert \le 1 $.   We then denote as $ \mathbf {BF}$ the best constant in 
the inequality 
\begin{equation}\label{e.BF}
\bigl\lvert 
 \sum_{I \;:\; \pi _{\mathcal F} I=F  } \sum_{\substack{J \;:\; J\Subset I \\ \pi _{\mathcal F} J=F}}  
 \mathbb E _{I_J} ^{\sigma } \Delta ^{\sigma } f \cdot \langle H _{\sigma } I_J , \Delta ^{\sigma } _{J} g  \rangle _{w}
\bigr\rvert \le \mathbf {BF} \{ \sigma (F) ^{1/2} + \lVert f\rVert_{\sigma }\} \lVert g\rVert_{w}
\end{equation}
where $ f\in BF _{\mathcal F} (F) $, and $ g$ is $ \mathcal F$-adapted to $ F$.  One must note that the 
two terms $ \sigma (F) ^{1/2} $ and 
$  \lVert f\rVert_{\sigma }$ on the right above are in general incomparable.  We refer to this as the \emph{bounded fluctuation condition}. 

This condition is  a consequence of the   boundedness of the form $ B _{\Subset}$, a fact which is not hard, and is proved below.\footnote{But it is not known to us that the bounded fluctuation condition is a consequence of the norm boundedness of the Hilbert transform.}
The role of the constant one in the inequalities 
$ \mathbb E ^{\sigma } _{I} \lvert  f\rvert \le 1 $ is immaterial.  It can be replaced by any fixed constant.  Indeed, if the measure 
$ \sigma $ is doubling, we could replace $ 1$ by a constant depending only on the doubling constant, then the bounded fluctuation condition reduces to the function being in $ L ^{\infty }$.

The following Theorem summarizes much of the content of this paper.  

%%%%%%%%%%%%%%%%%%%%%%%%%%%%%% THEOREM THEOREM THEOREM
\begin{theorem}\label{t.<} The following inequalities and their duals hold, for any pair of weights $ w, \sigma $ which 
	do not share a common point mass. 
\begin{gather}\label{e.<1}
\mathbf B _{\Subset}  \lesssim  \mathbf H + \mathbf F + \mathbf {BF} \,,  
\\ \label{e.<2}
\mathbf F   \lesssim \boldsymbol \Psi\,, \quad \textup{and } \quad \mathbf {BF} \lesssim \mathbf H + \boldsymbol \Psi \,, 
\\ \label{e.<3}
\mathbf F + \mathbf {BF}  \lesssim   \sqrt {\mathbf A_2} +  \mathbf W + \mathbf B _{\Subset}  \,. 
\end{gather}
\end{theorem}
%%%%%%%%%%%%%%%%%%%%%%%%%%%%%% THEOREM THEOREM THEOREM

 In particular, Theorem~\ref{t.main} is a corollary to the first two inequalities above, and their duals.  The interest in 
 \eqref{e.<3} is that it shows that the new side conditions, of functional energy and bounded fluctuation, are implications 
 of the proof strategy, namely the assumption that the bilinear form $ B _{\Subset} (f, \phi )$ is bounded.  Concerning \eqref{e.<2}, 
 the side condition controls the functional energy inequality by a straightforward argument, but the control of 
 the bounded fluctuation term is a deep argument, \S\ref{s.43}, initiated in \cite{1003.1596}.

\begin{question} \label{q2} For a pair of weights $ (w, \sigma )$, do any of these  inequalities hold? 
\begin{gather} \label{e.F?}
	\mathbf F  \lesssim \mathbf B \,,
\\ 
	\mathbf F  \lesssim  \sqrt{\mathbf A_2} + \mathbf H \,, 
\\ 
\mathbf {BF} \lesssim \mathbf B\,,  
\\
\label{e.BF?}
\mathbf {BF}  \lesssim  \sqrt{\mathbf A_2} + \mathbf H \,. 
\end{gather}
\end{question}

Note that the condition of functional energy is only about non-negative $ f$, and the `energy' of the weight $ w$.  
It is arguably an acceptable hypothesis to add to Conjecture~\ref{j.2wtH}; unfortunately, neither functional 
energy nor bounded fluctuation conditions admit an intrinsic formulation.  
The inequality on bounded fluctuation goes to the heart of the conjecture.

Finally, we indicate a new characterization of the two weight problem when just one weight is doubling.  This 
should be compared with the results of \cite{0805.0246}, which address maximal truncations, and also 
contrasts with a characterization in \cite{1003.1596} when \emph{both} weights are doubling.

%%%%%%%%%%%%%%%%%%%%%%%%%%%%%% THEOREM THEOREM THEOREM
\begin{theorem}\label{t.doubling} Let $ (w, \sigma )$ be a pair of weights with $ \sigma $ doubling.  Then, the two weight 
	inequality \eqref{e.H2} holds if and only if these constants are finite. 
	\begin{equation*}
		\mathbf A_2\,,\ \mathbf H\,,\ \mathbf H _{\ast }\,,\ \mathbf F\,,\ \mathbf {BF} < \infty \,. 
\end{equation*}
\end{theorem}
%%%%%%%%%%%%%%%%%%%%%%%%%%%%%% THEOREM THEOREM THEOREM

%%%%%%%%%%%%%%%%%%%%%%%%%%%%%% PROOF PROOF PROOF
\begin{proof}
	As $ \sigma $ is doubling, there is a constant $ c$ so that for any  interval $ I$, and any subinterval $ I'$ of length 
	$ \frac 14$ of $ I$, it holds that $ \sigma  (I')\ge c \sigma  (I)$.  From this, it follows that $ \mathsf E (\sigma ,I)\ge  c/4$.
	Namely, the energy of any interval is strictly bounded away from zero.  
	Assuming  the finiteness of $ \mathbf A_2$, $ \mathbf H _{\ast } $, as we may do in both directions of the argument,
	one may easily verify that the \emph{dual} Dini condition holds, that is $ \boldsymbol \Psi _{\ast } < \infty $.   
	(In fact, the pivotal condition of Nazarov-Treil-Volberg holds, as follows from the energy  condition \eqref{e.EnergyConstant}, which is necessary from $ \mathbf A_2$ and $ \mathbf H$.)  
	
	Assuming that the Hilbert transform is bounded, we necessarily have the finiteness of the $  A_2$ and 
	testing constants.  Therefore, the dual bilinear form is bounded, $ \mathbf B _{\Supset} < \infty $, hence 
	$ \mathbf B _{\Subset}$ is also finite, bounding $  \mathbf F $ and $\mathbf {BF}$, as claimed. 
	
	In the reverse direction, the finiteness of $ \mathbf A_2$, $ \mathbf H _{\ast } $ 
	and $ \boldsymbol \Psi _{\ast }$ implies 
	the boundedness of $ \mathbf B _{\Supset}  $, and the additional assumptions on functional energy $ \mathbf F$ and 
	bounded fluctuation $ \mathbf {BF}$ imply the boundedness of $  \mathbf B _{\Subset} $, hence the Hilbert transform 
	is bounded.    
\end{proof}
%%%%%%%%%%%%%%%%%%%%%%%%%%%%%% PROOF PROOF PROOF

%%%%%%%%%%%%%%%%%%%%%%%%%%%%%% SECTION  SECTION SECTION
%%%%%%%%%%%%%%%%%%%%%%%%%%%%%% SECTION  SECTION SECTION 
\section{The Splitting of the Operator} %\label{s.}

We expand the full bilinear form $ B (f,g) :=  \langle H _{\sigma} f, \phi  \rangle _{w}$ according  to the weighted Haar basis.
For the proof, we will take some (large) interval $ I ^{0}$, and assume that $ f $ and $ \phi $ are supported on $ I ^{0}$. 
Note that by the testing hypothesis,  
\begin{equation*}
\bigl\lvert 	\mathbb E ^{\sigma } _{I ^{0}} f \langle H _{\sigma } (I ^{0}), \phi  \rangle _{w}\bigr\rvert 
\le  \mathbf H \lvert 	\mathbb E ^{\sigma } _{I ^{0}} f \rvert \sigma (I ^{0}) ^{1/2} \lVert \phi \rVert_{w} \,. 
\end{equation*}
The dual inequality also holds, so we are free to assume that $ f$ and $ \phi $ have respective means zero, and hence are 
in the closed linear span of the (good) Haar functions.  

In the first generation, there are three terms,  which are  largely `below diagonal',  `diagonal', and `above diagonal' parts.  
\begin{align}
\langle H _{\sigma} f , \phi  \rangle_{w} 
& = B_{1,1} (f,\phi ) + B_{1,2} (f, \phi ) + B_{1,3} (f, \phi ) \, ,
\\
B_{s,t}(f, \phi )  
& := 
\sum_{ (I,J) \in \mathcal P _{s,t}}
\langle H _{\sigma} \Delta ^{\sigma}_I f , \Delta ^{w} _{J} \phi  \rangle _{w}  \, ,
\\ \label{e.12}
\mathcal P _{1,2} 
& := 
\bigl\{ (I,J) \;:\;  2 ^{-r} \lvert  I\rvert\le \lvert  J\rvert\le 2 ^{r} \lvert  J\rvert    \bigr\}
\\ \label{e.13}
\mathcal P _{1,3}
&:=
\bigl\{ (I,J) \;:\;   \lvert  J\rvert < \lvert  I\rvert \}\,. 
\end{align}
The term $ B_{1,1} $ is dual to  $ B_{1,3}$, so we do not explicitly define it here, as we will 
concentrate on $ B_{1,3}$.   

The diagonal term is  straightforward to control, and in \S\ref{s.12}, we will show 
\begin{equation} \label{e.12<}
\bigl\lvert  B_{1,2} (f ,\phi )\bigr\rvert \lesssim 
(\mathbf A_2 + \mathbf W) \lVert f\rVert_{\sigma} \lVert \phi \rVert_{w} \,. 
\end{equation}
We shall follow this pattern of postponing certain estimates that are 'routine' till a later section, preferring to 
pass to the more delicate parts of the decomposition, which will always have the larger  second indices. 

\medskip 

We concern ourselves with the term $ B_{1,3}$ defined in \eqref{e.21}.   And we right away 
split it into an `far away',  `local', and  `inside' part, defined as follows. Set 
\begin{align}
B_{1,3} &= B_{2,1}+ B_{2,2}+ B_{2,3} 
\\ \label{e.21}
\mathcal P _{2,1} 
& := \bigl\{ (I,J) \in \mathcal P _{1,3} \;:\;    3 I \cap J= \emptyset  \bigr\}\,,
\\ \label{e.22}
\mathcal P _{2,2} 
& := \bigl\{ (I,J) \in \mathcal P _{1,3} \;:\;     J \subset 3I \backslash I  \bigr\}\, ,
\\  \label{e.23}
\mathcal P _{2,3} 
& := \bigl\{ (I,J) \in \mathcal P _{1,3} \;:\;     J \Subset  I  \bigr\}\,.
\end{align}
In \S\ref{s.21} and \S\ref{s.22}, we will show that these two terms are also controlled by the $ A_2$ constant. 
\begin{equation}\label{e.21<}
\bigl\lvert B_{2,1} (f, \phi )\bigr\rvert+\bigl\lvert B_{2,2} (f, \phi )\bigr\rvert 
\lesssim 
\mathbf A_2 \lVert f\rVert_{\sigma} \lVert \phi \rVert_{w} \,. 
\end{equation}

Concerning the term $ B_{2,3}$, we will make this further decomposition.  
For the pairs of intervals $ (I,J)$ in question, we have $ J\subsetneq I$. Recall that 
$ I _{J}$ is the child of $ I$ that contains $ J$.  Now, the argument of the Hilbert transform 
is $ \Delta _{I} ^{\sigma} f$, which is constant on the two children of $ I$, namely $ I_J$ and $ I \backslash I_J$.  
This permits us to write 
\begin{align}
B_{2,3} &= B_{3,1} + B_{3,2}, 
\\ \label{e.31}
B_{3,1} (f, \phi ) &= 
\sum_{(I,J)\in \mathcal P _{2,3}}  
\mathbb E _{I \backslash I_J} ^{\sigma} \Delta^{\sigma} _{I} f \cdot  \langle H _{\sigma} (I - I_J), \Delta ^{w} _{J} \phi  \rangle_{w}\,  , 
\\ \label{e.32}
B_{3,2} (f, \phi ) &= 
\sum_{(I,J)\in \mathcal P _{2,3}}  
\mathbb E _{ I_J} ^{\sigma} \Delta^{\sigma} _{I} f \cdot \langle H _{\sigma}  I_J, \Delta ^{w} _{J} \phi  \rangle_{w}\,. 
\end{align}
We will show in \S\ref{s.31} that we have 
\begin{equation}\label{e.31<}
\bigl\lvert B_{3,1} (f, \phi )\bigr\rvert \lesssim 
\mathbf A_2 \lVert f\rVert_{\sigma} \lVert \phi \rVert_{w} \,. 
\end{equation}

The bilinear form $ B_{3,2} $ is the form $ B _{\Subset} $ of \S\ref{s.splitting}, and this is the notation that we will use below.
Our considerations to this point,  together with their duals, completes the proof of Theorem~\ref{t.preSplit}. 

%%%%%%%%%%%%%%%%%%%%%%%%%%%%%% SECTION  SECTION SECTION
%%%%%%%%%%%%%%%%%%%%%%%%%%%%%% SECTION  SECTION SECTION 
\section{The Calder\'on-Zygmund Corona} \label{s.corona} 

This section will be devoted to a proof of the inequality \eqref{e.<1}, namely that the side conditions of 
functional energy and bounded fluctuation can be used to control the bilinear form $ B _{\Subset}$.  This is the 
first of the two important corona arguments in the paper.  The reader should recall the definition of $ f$-stopping intervals $ \mathcal F$
in Definition~\ref{d.fStopping}. 

%%%%%%%%%%%%%%%%%%%%%%%%%%%%%% REMARK REMARK REMARK
\begin{remark}\label{r.standardArgument} The intervals $ {\mathcal F}$ are the standard construct in proving  paraproduct style arguments, moreover the identification and control of paraproducts is  an essential part of the two-weight problem.  
	Thus, it is natural to incorporate these intervals into the proof at an early stage. Indeed, if this step is not taken, nuanced 
	Carleson measure estimates are  needed. 
\end{remark}
%%%%%%%%%%%%%%%%%%%%%%%%%%%%%% REMARK REMARK REMARK

%%%%%%%%%%%%%%%%%%%%%%%%%%%%%%  DEFINITION DEFINITION DEFINITION
\begin{definition}\label{d.coronaPairs}[The Calder\'on-Zygmund  Corona Decomposition]
For $ F \in \mathcal F $, we say that the pair of intervals $ (I,J) \in \mathcal P _{2,3} $ are in $ \mathcal C (F)$ if and only if 
$ \pi _{\mathcal F} J=F$.  This definition only depends upon $ J$. 
We set $ \mathcal C_o (F)$ to be those pairs $ (I,J) \in \mathcal C (F)$ such that 
$ \pi _{\mathcal F}  I_J= F$. Note the dependence of this definition on the \emph{pair} $ (I,J)$. 
And, let $ \mathcal C^o (F)= \mathcal C (F) \backslash \mathcal C_o (F)$.  
Define associated projections 
\begin{equation}\label{e.PwS}
P ^{w } _{F} \phi := \sum _{J \;:\; \pi _{\mathcal F} J= F} \Delta ^{w } _{J} \phi \,. 
\end{equation}
Note that the latter projections are pairwise $ L ^2 (w )$-orthogonal in $ F\in \mathcal F$,  and we have 
\begin{equation}\label{e.PwSum}
\sum_{F\in \mathcal F} \lVert P ^{w } _{F} \phi\rVert_{w } ^2  \le \lVert \phi \rVert_{w } ^2 \,. 
\end{equation}
We use a similar, but distinct, notation   
$
P ^{\sigma  } _{F} f := \sum _{I \;:\; \pi _{\mathcal F} I _{\pm}= F}  \Delta ^{\sigma } _{I} f 
$.
Here, we sum over all $ I$ so that a dyadic child of $ I$ has $ \mathcal F$-parent $ F$. These projections are not orthogonal in $ F$, but nevertheless satisfy a variant of \eqref{e.PwSum} that we will need.  
\end{definition}
%%%%%%%%%%%%%%%%%%%%%%%%%%%%%%  DEFINITION DEFINITION DEFINITION

The \emph{(Calder\'on-Zygmund) corona decomposition} of the bilinear form $ B_{\Subset}$ is then based upon the $ f$-stopping intervals, hence non-linear in nature.
\begin{align}
B_{\Subset} (f, \phi ) &:= \sum_{F \in \mathcal F}
\sum_{t=1} ^{3}
B_{t} (f, \phi ; F) \,,
\\ \label{e.41}
B_{1} (f ,\phi; F ) & := 
\sum_{ (I,J) \in \mathcal C^o (F)}  \mathbb E _{F}  
\Delta ^{\sigma } _{I} f  \cdot 
\langle H _{\sigma } (I_F \backslash F) ,  \Delta^{w} _{J} \phi   \rangle _{w } \,,
\\ \label{e.42}
B_{2} (f, \phi ;F ) &:=  
\sum_{ (I,J) \in \mathcal C^o (F)}  \mathbb E _{F} \Delta ^{\sigma } _{I} f  \cdot 
\langle H _{\sigma } (F), \Delta^{w} _{J} \phi  \rangle _{w } \,, 
\\ \label{e.43}
B_{3} (f, \phi ; F)   
&:= \sum_{(I,J) \in \mathcal C_o (F)}  \mathbb E ^{\sigma } _{I_J} \Delta ^{\sigma } _{I} f \cdot 
\langle H _{\sigma } (I_J), \Delta ^{w } _{J} \phi  \rangle _{w } \,. 
\end{align}
Let us argue that we have equality above. 
The term $ B_{3}$ is the only one that depends upon $ \mathcal C _{o}$, and it will be further decomposed below.  
The remaining two terms depend upon the complementary part of the corona $\mathcal C$. 
For $ (I,J) \in \mathcal C^o$, note that $ I_J $ strictly contains $ F$,    $ \Delta _{I} f$ is constant on $ I_J\supset F$, hence 
$ \mathbb E _{F} \Delta ^{\sigma } _{I} f =  \mathbb E _{I_J} \Delta ^{\sigma } _{I} f  $.   
And, we have written $ I_J = I_F = F + (I_F \backslash F)$ to get the two terms $ B_{1}$ and $ B_{2}$.

%%%%%%%%%%%%%%%%%%%%%%%%%%%%%% SUBSECTION SUBSECTION SUBSECTION SUBSECTION
 %%%%%%%%%%%%%%%%%%%%%%%%%%%%%% SUBSECTION SUBSECTION SUBSECTION SUBSECTION 
 \subsection{The Term $ B_3$}%\label{ss.}
 
 We will show in \S\ref{s.43} that we have the inequality 
 \begin{equation}\label{e.b3<}
 	 \bigl\lvert B _{3} (f, \phi ; F)\bigr\rvert \lesssim 
 	 \{\mathbf H + \boldsymbol \Psi \} \bigl\{ \gamma (F) \sigma (F) ^{1/2} + \lVert P ^{\sigma } _{F} f \rVert_{\sigma } \bigr\} 
 	 \lVert P ^{w} _{F} \phi \rVert_{w}
\end{equation}
An application of Cauchy-Schwartz, and (quasi)-orthogonality will complete the estimate of this term.

%%%%%%%%%%%%%%%%%%%%%%%%%%%%%% SUBSECTION SUBSECTION SUBSECTION SUBSECTION
 %%%%%%%%%%%%%%%%%%%%%%%%%%%%%% SUBSECTION SUBSECTION SUBSECTION SUBSECTION 
 \subsection{The Term $ B _{2}$}%\label{ss.}
We claim the  estimate 
\begin{align} \label{e.42<}
\lvert B_{2} (f, \phi ; F)\rvert & \lesssim  \mathbf H 
 \gamma (F) \sigma (F) ^{1/2}  
 \lVert P ^{w } _{F} \phi\rVert_{w } \,, 
\end{align} 
In view of the quasi-orthogonality condition \eqref{e.F<} and \eqref{e.PwSum},  a trivial application of Cauchy-Schwartz will complete the  estimate of this term.  Namely, we have 
\begin{align*}
	\sum_{F\in \mathcal F}
	\lvert B_{2} (f, \phi ; F)\rvert & \lesssim  \mathbf H 
	\Biggl[  \sum_{F\in \mathcal F} \gamma (F) ^2 \sigma (F) \times 
	\sum_{F\in \mathcal F}  \lVert P ^{w } _{F} \phi\rVert_{w } ^2 
	\Biggr] ^{1/2} 
	\\
	& \lesssim  \mathbf H  \lVert f\rVert_{\sigma } \lVert \phi \rVert_{w} \,. 
\end{align*}

The proof of \eqref{e.42<} is quickly obtained.  
We estimate, using the telescoping property of martingale differences,  
\begin{align*}
\lvert  B_{2} (f, \phi )\rvert  & \le 
\sum _{J \;:\; \pi _{\mathcal F} J = F} 
\Bigl\lvert  \sum_{I \;:\; (I,J) \in \mathcal C^o (F)} 
\mathbb E _{F} \Delta ^{\sigma } _{I} f 
\Bigr\rvert \cdot \lvert\langle  H _{\sigma } F , \Delta ^{w} _{J} \phi  \rangle_{w}\rvert
\\
& =  \bigl\lvert \mathbb E ^{\sigma } _{F} f \bigr\rvert  \lVert H _{\sigma  } (F)\rVert_{w }  \lVert \mathsf P ^{w} _{F} \phi \rVert_{w} 
\\
& \lesssim  \mathbf H \gamma (F) \sigma (F) ^{1/2}  \lVert \mathsf P ^{w} _{F} \phi \rVert_{w}  \,. 
\end{align*}
 The expression $ \mathbb E _{F} \Delta ^{\sigma } _{I}f$ arises above since for $ (I,J) \in \mathcal C ^{0} (F)$, it holds  that $F\subsetneq I_J $. Hence, the sum of martingale differences can be summed exactly as above.

For $ B_{3} (f, \phi ; F)$, the definition of bounded fluctuation in \eqref{e.BF} was constructed for this term.  
Namely, the function $ (C \gamma (F)) ^{-1} f \cdot F$ is in $  BF _{\mathcal F} (F)$.  
The function $ P ^{w} _{F} \phi $ is $ \mathcal F$-adapted to $ F$, whence 
\begin{equation*}
\bigl\lvert B_{3} (f, \phi ; F)\bigr\rvert \lesssim \mathbf {BF} 
\{ \gamma (F) \sigma (F) ^{1/2} + \lVert P ^{\sigma } _{F} f\rVert_{\sigma } \} 
\lVert P ^{w} _{F} \phi \rVert_{w} 
\end{equation*}
An application of Cauchy-Schwartz, and  appeal to (quasi-)orthogonality  will complete the analysis of this term. 

%%%%%%%%%%%%%%%%%%%%%%%%%%%%%% SUBSECTION SUBSECTION SUBSECTION SUBSECTION
%%%%%%%%%%%%%%%%%%%%%%%%%%%%%% SUBSECTION SUBSECTION SUBSECTION SUBSECTION 
\subsection{The Term $ B_{1}$.}\label{s.41}

The analysis of \eqref{e.41} is based upon the functional energy condition, and leads to this estimate: 
\begin{equation}\label{e.41<}
\bigl\lvert B_{1} (f, \phi )\bigr\rvert \lesssim 
\mathbf F  \lVert f\rVert_{\sigma } \lVert \phi \rVert_{w } \,. 
\end{equation}
The following lemma records a  monotonicity property for the Hilbert transform, and a property involving the 
Poisson integral.

\begin{lemma}[Monotonicity Property]
\label{mono}Suppose that $\nu $ is a signed measure, and $\mu $ is a
positive measure with $\mu \geq \left\vert \nu \right\vert $, both supported
outside an interval $I$.  Let $J\subset  J ^{\ast} \Subset I$.  Then it holds that 
\begin{equation}\label{e.mono+P}
	\left\vert \left\langle H\nu ,h^w_J \right\rangle _{w }\right\vert \leq 	\left\langle H\mu ,h^w_J \right\rangle _{w } 
\end{equation}
In addition, we have the estimate 
\begin{equation}\label{e.Elower}
	\mathsf P ( \mu , J  ^{\ast} )   \bigl\lvert \bigl\langle  \frac x {\lvert  J ^{\ast} \rvert }, h^w_J  \rangle _{w}\bigr\rvert 
\lesssim 
	\left\langle H\mu ,h^w_J \right\rangle _{w } 
  + \Bigl[\frac {\lvert  J\rvert } {\lvert  I\rvert } \Bigr] ^{1- \epsilon } P (\mu ,J) \sqrt {w (J)}
\end{equation}  
\end{lemma}

The function $ H \mu $ will be monotonically decreasing on $ J$, and we have chosen the definition of the 
Haar functions so that $ \left\langle H\mu ,h^w_J \right\rangle _{w } $ is non-negative, while $ \langle x,h ^{w} _{J} \rangle _{w}$ 
is negative.  

\begin{proof}  This argument is special to the Hilbert transform. 
Let $J_{-}=J\cap \left( -\infty ,c\right) $ and $J_{+}=J\cap \left( c,\infty
\right) $. We may renormalize the Haar function $ h ^{w} _{J}$ so that 
\begin{equation*}
\int_{J_{-}}\left\vert h^w_J \right\vert dw =\int_{J_{+}}\left\vert
h^w_J \right\vert dw =1.
\end{equation*}%
Then we have%
\begin{align}
\left\langle H\nu ,h^w_J \right\rangle _{w } &=\int_{J_{+}}H\nu
\left( x\right) h^w_J \left( x\right) dw \left( x\right)
+\int_{J_{-}}H\nu \left( x\right) h^w_J \left( x\right) dw \left(
x\right) \\
&=\int_{J_{+}}H\nu \left( x\right) \left\vert h^w_J \left( x\right)
\right\vert dw \left( x\right) -\int_{J_{-}}H\nu \left( x^{\prime
}\right) \left\vert h^w_J \left( x^{\prime }\right) \right\vert dw
\left( x^{\prime }\right) \\
&=\int_{J_{+}}\int_{J_{-}}\left[ H\nu \left( x\right) -H\nu \left(
x^{\prime }\right) \right] \left\vert h^w_J \left( x^{\prime }\right)
\right\vert dw \left( x^{\prime }\right) \left\vert h^w_J \left(
x\right) \right\vert dw \left( x\right) \\
\label{e.cvb}
&=\int_{J_{+}}\int_{J_{-}}\int_{\mathbb{R}\setminus J}\frac{x-x^{\prime }}{%
\left( y-x\right) \left( y-x^{\prime }\right) }d\nu \left( y\right)
\left\vert h^w_J \left( x^{\prime }\right) \right\vert dw \left(
x^{\prime }\right) \left\vert h^w_J \left( x\right) \right\vert dw
\left( x\right) ,
\end{align}%
and since $\frac{x-x^{\prime }}{\left( y-x\right) \left( y-x^{\prime
}\right) }\geq 0$ for $y\in \mathbb{R}\setminus J$ and $x\in J_{+}$ and $%
x^{\prime }\in J_{-}$, we have%
\begin{align*}
\left\vert \left\langle H\nu ,h^w_J \right\rangle _{w }\right\vert
&\leq \int_{J_{+}}\int_{J_{-}}\int_{\mathbb{R}\setminus J}\frac{x-x^{\prime
}}{\left( y-x\right) \left( y-x^{\prime }\right) }d\mu \left( y\right)
\left\vert h^w_J \left( x^{\prime }\right) \right\vert dw \left(
x^{\prime }\right) \left\vert h^w_J \left( x\right) \right\vert dw
\left( x\right) \\
&=\left\langle H\mu ,h^w_J \right\rangle _{w },
\end{align*}%
where the last equality follows from the previous display with $\mu $ in
place of $\nu $.  This concludes the first half of \eqref{e.mono+P}. 

For the second estimate \eqref{e.Elower},  we will make a first order Taylor polynomial approximation of 
$ H   \mu $ on the interval $ J$. Let us denote the derivative  by $ D$, 
and for $ x\in J$ note that 
\begin{equation*}
	D H \mu (x) = - \int \frac1  { (x-y) ^2 } \; \mu (dx)\,, \quad 
	D ^2 H \mu (x)  =  \int \frac1 { (x-y) ^3 } \; \mu (dx)\,. 
 \end{equation*}
The point here is that the second derivative is somewhat small. 
From this, we can write, letting $ c_J$ be the center of $ J$,  
\begin{align*}
\bigl\lvert 	H \mu (x) - H \mu (c_J)  -  (x-c_J) H \mu (c_J) \bigr\rvert 
&\le \sup _{x \in J} (x - c_J) ^2 \lvert  D ^2 H \mu (x)\rvert 
\\
& \lesssim  \frac {\lvert  J\rvert } {\textup{dist} (\partial J, I)} P (\mu ,J)\,, \qquad x\in J\,.  
\end{align*}

When we are estimating the inner product with a Haar function, constants are immaterial, therefore, using the fact that $ J$ is good, 
we can write 
\begin{align*}
	\bigl\lvert \langle H \mu , h ^{w} _{J} \rangle _{w} - 
	D H \mu (c_J) \bigl\langle  x ,  h ^{w} _{J} \bigr\rangle_{w}
	\bigr\rvert
	& \lesssim \Bigl[\frac {\lvert  J\rvert } {\lvert  I\rvert } \Bigr] ^{1- \epsilon } P (\mu ,J) \sqrt {w (J)} 
\end{align*}
Finally, one uses $	D H \mu (c_J)  \lesssim \lvert  J\rvert ^{-1} P (\mu ,J)  $.  This finishes the argument. 
\end{proof}

 Returning to the analysis of $ B_1$, write  
\begin{equation*}
	\widetilde f _{F} := \sum_{I \;:\; F\subsetneq I}  \mathbb E _{I_F} \Delta ^{\sigma } _{I} f \cdot (I_F- F) \,, 
	\qquad 
	 \overline f := \sum_{F\in \mathcal F} \gamma (F) \cdot F \,. 
\end{equation*}
Note that  $ \lvert \widetilde f _{F}\rvert \lesssim \overline  f$.  
Let $ \mathcal J ^{^{\ast} } (F)$ be the maximal  $ (\epsilon ,r)$-good intervals $ J ^{\ast} \Subset F$. 
Applying the Lemma, we have 
\begin{align*}
\lvert  B_{1} (f, \phi ;F)\rvert 
& \le \sum_{J ^{\ast} \in \mathcal J ^{\ast} (F)} \sum_{\substack{J \subset J ^{\ast} \\ \pi _{\mathcal F} J=F }} 
\lvert  \langle  H _{\sigma } \widetilde f _{F}, \Delta ^{w} _{J} \phi  \rangle _{w}\rvert 
\\
& \lesssim 
\sum_{J ^{\ast} \in \mathcal J ^{\ast} (F)} \sum_{\substack{J \subset J ^{\ast} \\ \pi _{\mathcal F} J=F }} 
P ( \overline  f  \sigma   (\mathbb R -F),  J ^{\ast} ) 
\bigl\lvert  \bigl\langle  \frac x {\lvert  J ^{\ast} \rvert },  \Delta ^{w} _{J} \phi  \bigr\rangle \bigr\rvert
\\
&=
\sum_{J ^{\ast} \in \mathcal J ^{\ast} (F)} \sum_{\substack{J \subset J ^{\ast} \\ \pi _{\mathcal F} J=F }} 
P ( \overline  f \cdot \sigma ,  J ^{\ast} ) 
\bigr\langle  \frac x {\lvert  J ^{\ast} \rvert },  \Delta ^{w} _{J}\overline   \phi  \bigr\rangle 
\end{align*}  
where $ \overline \phi $ is a obtained from $ \phi $ by an appropriate $ w$-Haar multiplier, chosen to make all 
the inner products above non-negative so that the absolute
values can be removed.   
By orthogonality of the projections $ P ^{w} _{F} $, and the definition of the functional energy condition, we see that the 
sum over $ F\in \mathcal F$ of this last expression verifies \eqref{e.41<}. 

To be specific, an operator $ T$ is a  \emph{$ w$-Haar multiplier} if it is of the form 
$ T \phi = \sum_{J\in \mathcal D } \varepsilon _J   \Delta ^{w} _{J} \phi $, with $ \lvert  \varepsilon _J\rvert=1 $. These operators are isometries on $ L ^2 (w)$.  The multiplier we need has 
$ \varepsilon _J = \operatorname {sgn} (\langle \phi , h ^{w} _{J} \rangle _{w})$, and $ \overline  \phi = T \phi $. 

%%%%%%%%%%%%%%%%%%%%%%%%%%%%%% SUBSECTION SUBSECTION SUBSECTION SUBSECTION
%%%%%%%%%%%%%%%%%%%%%%%%%%%%%% SUBSECTION SUBSECTION SUBSECTION SUBSECTION 
\section{Bounded Fluctuation and the Second Corona}  \label{s.43}

There are two estimates of the bounded fluctuation constant $ \mathbf{BF} $ that should be made, the easy estimate 
of $ \mathbf{BF} \lesssim   \mathbf B _{\Subset}$, and the difficult estimate of 
$ \mathbf{BF} \lesssim   \sqrt {\mathbf A_2}  + \boldsymbol \Psi $.  We turn to the second estimate, which is involved. 

Fix the data for the bounded fluctuation term.  $ F$ is an interval, and $ \textup{Child} (F)$ are the intervals inside $ F$; 
the function  $ f$ is of bounded fluctuation relative to this data, and $ \phi $ is adapted to $ \{F\}\cup \textup{Child} (F)$.  
We consider the difficult estimate, in which the Dini and testing  conditions dominate  the bounded fluctuation term.  Setting notation, we are  to show  \eqref{e.b3<}, which is the same as this estimate.  
\begin{align} 
B _{\textup{stop}} (f, \phi ) & := 
 \sum_{I \;:\; \pi _{\mathcal F} I=F}  \sum_{ \substack{J \;:\; J\Subset I \\ \pi _{\mathcal F} J=F}} 
  \mathbb E _{I_J} ^{\sigma } \Delta ^{\sigma } f \cdot \langle H _{\sigma } I_J , \Delta ^{w} _{J} \phi  \rangle _{w}
\\  
 \label{e.43<} 
\bigl\lvert B_{\textup{stop}} (f, \phi)   \bigr\rvert
&\lesssim \bigl\{  \mathbf H + \boldsymbol \Psi  \}  
\bigl\{\gamma (F) \sigma (F) ^{1/2} + \lVert f\rVert_{\sigma } \bigr\}
\lVert  \phi \rVert_{w} \,. 
\end{align}
The origins of this argument are derived from \cite{1003.1596}, as modified in \cite{1001.4043}; these papers refer to 
this term as the \emph{stopping term.} We will again 
find simplifications by the use of the Calder\'on-Zygmund corona.   We define here the Dini corona.

%%%%%%%%%%%%%%%%%%%%%%%%%%%%%%  DEFINITION DEFINITION DEFINITION
\begin{definition}\label{d.stopping}  Let $ I_0\subset F$. 
We set $ \mathcal S (I_0)$  to be the maximal  subintervals $ S \subsetneq I_0$ such that
\begin{gather}\label{e.badEnergy}
\Psi_w(I_0, S) ^2  \ge 4  \boldsymbol \Psi ^2  \sigma (S) \,. 
\\
\Psi_w(I_0, S) ^2 
:= \sup \psi (s) ^{-2}  \sum_{ j =1 } ^{\infty } \sum_{k=1} ^{\infty } 
 \mathsf P (\sigma \cdot (I_0 - I_j ), I_{j,k}) ^2 \mathsf E (w , I_{j,k}) ^2 w (I_{j,k})
\end{gather}
where the supremum is formed over the various data that enter into Definition~\ref{d.side}, to wit: 
%%  ITEMIZE
\begin{itemize}
\item $ \{I_j \;:\; j\ge 1\}$ is a sub-partition of $ S$ into intervals; 
\item $ \{I _{j,k} \;:\; k\ge 1 \}$ is a sub-partition of $ I _{j}$ into good  intervals, 
\item  $ s\ge r$ is an integer and $ \lvert  I _{j,k}\rvert < 2 ^{-s} \lvert  I _{j}\rvert  $, for all $ j,k$. 
  \end{itemize}
%% ITEMIZE
We then set $ \mathcal  S := \bigcup _{s=0} ^{\infty } \mathcal S _s$, where $ \mathcal S_0= \{F\}$, and 
inductively, $ \mathcal S _{s+1} = \bigcup _{S\in \mathcal S _{s} } \mathcal S (S)$.  
\end{definition}
%%%%%%%%%%%%%%%%%%%%%%%%%%%%%%  DEFINITION DEFINITION DEFINITION

It is important to note that despite the assumption of the Dini energy condition, there is no \emph{a priori} upper 
bound of the quantity $ \Psi_w(I_0, S)$ in terms of $ \sigma (S)$.   We also have the following elementary estimate, but critical, 
\begin{equation} \label{e.S<}
\sum_{S\in \mathcal S (I_0)} \sigma (S) < \tfrac 14 \sigma (I_0) \,. 
\end{equation}
 We have by \eqref{e.PsiI}, 
\begin{align*}
4 \boldsymbol \Psi ^2  \sum_{S\in \mathcal S}\sigma (S)  
 \le 
 \sum_{S\in \mathcal S} \Psi_w(I_0 ,S)  ^2 \le \boldsymbol \Psi ^2  \sigma (I_0) \,. 
\end{align*}
The constant $ \boldsymbol \Psi  ^2 $ divides out, so that  \eqref{e.S<} holds.

The Dini corona decomposition of $ \mathcal C _o (F)$ is then  the collection of pairs  $ \mathcal B (S)$,
of those $ (I,J) \in \mathcal C_o (F)$ such that $ J$ has $ \mathcal S$-parent $ S$.   We further write $ 
\mathcal B (S) $ as the disjoint union of $ \mathcal B_o (S)\dot\cup \mathcal B ^{o} (S)$ 
where $ \mathcal B _{o} (S)$ consists of those pairs $ (I,J) \in \mathcal B (S)$,
where $ I_J \subsetneq S$. 
This definition is carefully crafted so that \eqref{e.PsiI} fails for $ I_J$ if $(I,J)$ is in $B_0(S)$. 

We then split the term $  B _{\textup{stop}} (f, \phi )$ up according to the corona.  The argument of the Hilbert transform is 
also split up. Here, it is a basic fact that for each $ J$, the function 
\begin{equation} \label{e.bJ}
b_ J := \sum_{I \;:\; (I,J) \in \mathcal B_o (F) \cup \mathcal B ^{o} (F)}  \mathbb E _{ I_J} ^{\sigma} \Delta^{\sigma} _{I} f  \cdot  I_J 
\end{equation}
is supported on $ F$, and has $ L ^{\infty }$ norm dominated by $ 2$.   
The Hilbert transform is applied to $ I_J$.  Let $ S$ be the $ \mathcal S$-parent of $ J$, We will write this as 
\begin{align*}
I_J   = 
\begin{cases}
S - (S - I_J) &  \textup{ $I_J \subsetneq  S $, equivalently, $ (I,J) \in \mathcal B _{o} (S)$,} 
\\
(I_J - S) +S &  \textup{$I_J \supset S $, equivalently, $ (I,J)\in \mathcal B ^{o} (S)$}\,. 
\end{cases}
\end{align*}
And this permits us to write 
\begin{align}
 B _{\textup{stop}} (f, \phi) &= \sum _{S\in \mathcal S} B _{1} (f, \phi ; S) + B _{2} (f, \phi ; S) - B _{3} (f, \phi ; S) 
  \\ \label{e.51} 
B _{1} (f, \phi ; S) &:= 
\sum_{(I,J) \in \mathcal B(S)}
  \mathbb E ^{\sigma } _{I_J} \Delta ^{\sigma } _{I} f  \cdot \langle H _{\sigma } (S),  \Delta ^{w } _{J} \phi \rangle _{w } \,
  \\ \label{e.52}
B _{2} (f, \phi ; S) &:= 
 \sum_{(I,J) \in \mathcal B ^{o} (S) } 
  \mathbb E ^{\sigma } _{I_J} \Delta ^{\sigma } _{I} f  \cdot \langle H _{\sigma } (I_J-S) ,  \Delta ^{w } _{J} \phi \rangle _{w } \,
\\
\label{e.53}
B _{3} (f, \phi ; S) &:= 
 \sum_{(I,J) \in \mathcal B_o (S) } 
  \mathbb E ^{\sigma } _{I_J} \Delta ^{\sigma } _{I} f  \cdot \langle H _{\sigma } (S-I_J),  \Delta ^{w } _{J} \phi \rangle _{w } \,
\end{align}

%%%%%%%%%%%%%%%%%%%%%%%%%%%%%% SECTION  SECTION SECTION
%%%%%%%%%%%%%%%%%%%%%%%%%%%%%% SECTION  SECTION SECTION 
\subsection{The control of $ B_{3}$} %\label{s.}
We take up the most delicate case of $  B_{3} (f, \phi ; S)$, showing that 
\begin{equation}\label{e.53<}
\lvert   B_{3} (f, \phi ; S)\rvert \lesssim \boldsymbol \Psi   \lVert  P ^{\sigma } _{S} f \rVert_{\sigma } 
\lVert P _{S} ^{w} \phi \rVert_{w} \,, \qquad S\in \mathcal S \,.  
\end{equation}
Here, the projection $  P ^{\sigma } _{S} f  $ is onto the span of the Haar functions $ h ^{\sigma } _{I}$ such that 
a child of $ I$ has $ \mathcal S$-parent $ S$, and $ P ^{w} _{S} \phi $ has an analogous definition.  Note that 
projections $ P ^{w} _{S}$ are pairwise orthogonal, while a given Haar function can only contribute to at most  two 
projections $ P ^{\sigma } _{S}$.  This and application of Cauchy-Schwartz will show that 
\begin{equation*}
\sum_{S\in \mathcal S} 
\lvert   B_{3} (f, \phi ; S)\rvert  \lesssim  \boldsymbol \Psi  \lVert    f \rVert_{\sigma } 
\lVert \phi \rVert_{w} \,, 
\end{equation*}
which is as required in \eqref{e.43<}.  

In the main estimate, we hold the relative lengths of $ I$ and $ J$ constant. It holds that 
\begin{align}\label{e.53s<}
\lvert B_{3,s} (f, \phi ; S)\rvert & := \Bigl\lvert 
\sum_{\substack{ (I,J) \in {\mathcal B_{0} (S)} \\  2 ^{s}\lvert  J\rvert = \lvert  I\rvert}} 
\mathbb E _{ I_J} ^{\sigma} \Delta^{\sigma} _{I} f \langle H _{\sigma}  (S-I_J), \Delta ^{w} _{J} \phi  \rangle_{w}
\Bigr\rvert 
\\&\lesssim \boldsymbol \Psi  \psi (s)
\lVert P ^{\sigma } _{S} f \rVert_{\sigma } \lVert P ^{w} _{S} \phi \rVert_{w}\,, \qquad  s>r\,. 
\end{align}
The constants $ \boldsymbol \Psi $ and $ \psi (s)$ are as in Definition~\ref{d.side}, and in particular, $ \sum_{s} \psi (s) \le 1$.  
This is summed over $ s$ to finish the proof of \eqref{e.53<}.

To prove the inequality above, we use this observation. 
For any  choice of sign, 
\begin{equation} \label{e.Eip}
\bigl\lvert  \mathbb E ^{\sigma } _{I_\pm} \Delta ^{\sigma} _{I} f  \bigr\rvert 
= 
\lvert  \langle f, h ^{\sigma } _{I} \rangle _{\sigma } \rvert 
\bigl\lvert  \mathbb E ^{\sigma } _{I _{\pm}}  h ^{\sigma } _{I} \bigr\rvert 
\le 
\lvert  \langle f, h ^{\sigma } _{I} \rangle _{\sigma } \rvert \sigma (I _{\pm}) ^{-1/2} \,.
\end{equation}
This is the elementary property of the Haar functions of \eqref{e.h<}.  
We apply Cauchy-Schwartz in $ I$ to the expression below 
\begin{align*}
\lvert  B_{3,s} (f, \phi ; S)\rvert ^2 
& \lesssim 
\lVert P _{S} ^{\sigma } f \rVert_{\sigma } ^2 
\times 
\sum_{ I \;:\; \pi _{\mathcal S} I=S}  \sigma (I_J) ^{-1}  \Biggl[ 
\sum_{J \;:\; \substack{ (I,J) \in {\mathcal B_{1}(S)} \\  2 ^{s}\lvert  J\rvert = \lvert  I\rvert}}  
\bigl\lvert 
\langle H _{\sigma}  (S- I_J) , \Delta ^{w} _{J} \phi  \rangle_{w}
\bigr\rvert \Biggr] ^2 
\\
& \lesssim 
\lVert P _{S} ^{\sigma } f_1 \rVert_{\sigma} ^2  \cdot 
\lVert P _{S} ^{w } \phi  \rVert_{w} ^2 
\times  M _s ^2 
\\
M ^2_s   & := 
\sup _{ I \;:\; \pi _{\mathcal S} I=S} \sup _{\theta \in \{\pm\}}  \sigma (I _{\theta }) ^{-1} 
\sum_{J \;:\; \substack{ (I,J) \in {\mathcal B_o(S)} \\  2 ^{s}\lvert  J\rvert = \lvert  I\rvert \\ I_J= I_ \theta }} 
\langle H _{\sigma}  (S- I_ \theta  ) , h ^{w} _{J} \rangle_{w} ^2  \,. 
\end{align*}
Here,  to get the bound in terms of $ \lVert P _{F} ^{w } \phi  \rVert_{w} ^2 $, we use the fact that for 
fixed $ J \subset S$, there is a unique $ I$ so that $ J\subset I $, $ 2 ^{s} \lvert  J\rvert=\lvert  I\rvert   $, and $ (I,J) \in \mathcal B_o $. 
We turn our attention to $ M_s$. 
Applying  \eqref{e.mono+P},  and  the definition of the  $ \Psi $-functional in Definition~\ref{d.side}, that we have 
\begin{align*} \sigma (I _{\theta }) ^{-1}
\sum_{J \;:\; \substack{ (I,J) \in {\mathcal B_o(S)} \\  2 ^{s}\lvert  J\rvert = \lvert  I\rvert\\ I_J= I_ \theta }} 
\langle H _{\sigma}  (S- I_ \theta ) , h ^{w} _{J} \rangle_{w} ^2  
&\le \boldsymbol \Psi ^2  \psi (s) ^2  \sigma (I _{\theta }) ^{-1}  \Psi_w(F, I _{\theta })  \le 4 \psi (s) ^2 \boldsymbol \Psi ^2  \,.
\end{align*}
In the last inequality, it is decisive that the interval $ I_J \subsetneq S$, hence  \emph{fails} the condition  \eqref{e.badEnergy}. 

%%%%%%%%%%%%%%%%%%%%%%%%%%%%%% SUBSECTION SUBSECTION SUBSECTION SUBSECTION
 %%%%%%%%%%%%%%%%%%%%%%%%%%%%%% SUBSECTION SUBSECTION SUBSECTION SUBSECTION 
\subsection{The Control of $ B_{2}$.}%\label{ss.}
For $ S\in \mathcal S$, let $ P ^{w} _{S}$ be the projection onto the span of Haar functions $ h ^{w} _{J}$ with $ \pi _{\mathcal S} J=S$.  
By Lemma~\ref{mono}, there is a function $\overline  \phi $, a $ w$-Haar multiplier of $ \phi $, so that 
\begin{align*}
\lvert B_{2} (f, \phi ; S) \rvert
& \lesssim   \langle  H _{\sigma } (F-S),  P ^{w} _{S} \overline  \phi \rangle _{w} 
\\
& = 
  \langle  H _{\sigma } F,    P ^{w} _{S} \overline \phi \rangle _{w} 
-   \langle  H _{\sigma } S,     P ^{w} _{S} \overline \phi \rangle _{w} 
\end{align*}
Now for the first term on the right above,
\begin{align}
\sum_{S\in \mathcal S - \{F\}}  \langle  H _{\sigma } F,   P ^{w} _{S} \overline \phi \rangle _{w}  
& = 
\Bigl\langle H _{\sigma } F , \sum_{S\in \mathcal S - \{F\}}   P ^{w} _{S}\overline  \phi
\Bigr\rangle_{w}
\label{e.52<} \le \mathbf H \sigma (F) ^{1/2} \lVert P ^{w} _{F} \phi \rVert_{w} 
\end{align}
And, for the second term on the right above, 
\begin{align*}
\sum_{S\in \mathcal S} \lvert  \langle  H _{\sigma } S,    P ^{w} _{S} \overline \phi \rangle _{w} \rvert
& \le \mathbf H 
\sum_{S\in \mathcal S}  \sigma (S) ^{1/2} \lVert P ^{w} _{S} \phi \rVert_{w} 
\\
& \le \mathbf H 
\Biggl[
\sum_{S\in \mathcal S} \sigma (S) \times 
\sum_{S\in \mathcal S} \lVert P ^{w} _{S} \phi \rVert_{w}  ^2 
\Biggr] ^{1/2} 
\\
& \lesssim \mathbf H \sigma (F) ^{1/2}  \lVert P ^{w} _{F}\rVert_{w}  \,. 
\end{align*}
Here, we have appealed to the critical estimate \eqref{e.S<}.  This with \eqref{e.52<} completes the bound of 
$ B_{2} (f, \phi)$.

%%%%%%%%%%%%%%%%%%%%%%%%%%%%%% SUBSECTION SUBSECTION SUBSECTION SUBSECTION
 %%%%%%%%%%%%%%%%%%%%%%%%%%%%%% SUBSECTION SUBSECTION SUBSECTION SUBSECTION 
\subsection{The Control of $ B_{1}$.}%\label{ss.}
The bound for $ B_{1}$, as defined in \eqref{e.51}, is straightforward.  Recalling our observation that the   functions   $ b_J$ in 
\eqref{e.bJ} are bounded in $ L ^{\infty }$ by $ \gamma (F)$, one can appeal directly to the testing condition to see that 
\begin{align} \label{e.51<}
\lvert  B_{1} (f, \phi ; S)\rvert &  \lesssim 
\mathbf H  \cdot M \cdot    \sigma (S) ^{1/2} \lVert P ^{w} _{S} \phi \rVert_{w} \,, 
\\
M & := \sup _{J \;:\; \pi _{\mathcal S} J =S} 
\Bigl\lvert \sum_{ I \;:\; (I,J) \in \mathcal B(S)}
  \mathbb E ^{\sigma } _{I_J} \Delta ^{\sigma } _{I} f  \Bigr\rvert \,. 
\end{align}
But $ M \lesssim  \gamma (F)$.

Using the orthogonality of the projections $ P ^{w} _{S}$, and the condition  on the stopping intervals \eqref{e.S<}, one sees that 
\begin{align}    
\sum_{S\in \mathcal S} \lvert  B_{1} (f, \phi ; S)\rvert & \lesssim  
\mathbf H \gamma (F) 
\Biggl[ \sum_{S\in \mathcal S} \sigma (S)  \times 
\sum_{S\in \mathcal S} 
\lVert P ^{w} _{S} \phi \rVert_{w} ^2 
\Biggr] ^{1/2} 
\\
& \lesssim \mathbf H \gamma (F) \sigma (F) ^{1/2}  \lVert P ^{w} _{F} \phi \rVert_{w}  \,. 
\end{align}
This is as required by \eqref{e.43<}.

%%%%%%%%%%%%%%%%%%%%%%%%%%%%%% SUBSECTION SUBSECTION SUBSECTION SUBSECTION
 %%%%%%%%%%%%%%%%%%%%%%%%%%%%%% SUBSECTION SUBSECTION SUBSECTION SUBSECTION 
\subsection{A Second Estimate}%\label{ss.}
We have completed the proof of \eqref{e.43<}, and turn to the easy estimate $ \mathbf{BF} \lesssim \mathbf B _{\Subset} + \mathbf H$.  
Indeed, if we are given a function $ f$ and $ \phi $ with which we are to test the bounded fluctuation condition, 
note that the sum that appears in \eqref{e.BF} reduces to  $ B _{\textup{stop}} (f, \phi)$.   
But,  we have 
\begin{align*}
\bigl\lvert  
B _{\textup{stop}} (f, \phi) - B _{\Subset} (f, \phi )
\bigr\rvert 
& 
\le 
\Biggl\lvert 
\sum_{I \;:\; I\supsetneq F} \sum _{J \;:\; J\Subset F}
\mathbb E _{F} \Delta ^{\sigma } _{I} f \langle H _{\sigma } (I_J )  , g\rangle
\Biggr\rvert
\\
& 
\le \bigl\lvert  \mathbb E _{F} f \langle H _{\sigma } F , g\rangle
\bigr\rvert
+
\Biggl\lvert 
\sum_{I \;:\; I\supsetneq F}  \mathbb E _{F} \Delta ^{\sigma } _{I} f \langle H _{\sigma } (I_ { F}-F)  , g\rangle
\Biggr\rvert
\end{align*}
The first term is bounded by $\mathbf H  \sigma (F) ^{1/2} \lVert g\rVert_{w}$.  
The second term is zero, since $ f$ is supported on $ F$, hence 
$ \sum_{I \;:\; I\supsetneq F}  \mathbb E _{F} \Delta ^{\sigma } _{I} f \cdot  (I_ {F} - F) \equiv 0$.

\smallskip
We argue that  $ \mathbf H \lesssim \sqrt {\mathbf A_2} +  \mathbf W+ \mathbf B _{\Subset}$, 
which completes our proof of \eqref{e.<3}, that is $ \mathbf {BF} \lesssim  \sqrt {\mathbf A_2} +   \mathbf W+ \mathbf B _{\Subset}$.

Let us fix an interval $ I ^0$, and function $ \phi  \in L ^2 (w)$ supported on $ I^0$, for which we are to estimate $ \langle H _{\sigma } I^0, \phi  \rangle _{w}$ in terms of the $A_2 $ constant, the weak-boundedness constant, and the split form constant $ \mathbf B _{\Subset}$.      
By appealing to the weak-boundedness constant $ \mathbf W$, we can assume that $ \phi $ has $ w$-integral zero. 
We are also free to consider (random) dyadic grids $ \mathcal D$, with respect to which $ I ^{0}$ is dyadic.  It follows that we can take 
$ \phi $ in the linear span of $ \{h ^{w} _{J } \;:\; J \subset I ^{0},\  \textup{ $ J$ is good}\}$.

By appealing to Theorem~\ref{t.preSplit}, it suffices to consider the sum of the  two forms $ B _{\Supset} ( I ^{0}, \phi )+ B _{\Subset} ( I ^{0}, \phi )$. 
But $  B _{\Supset} ( I ^{0}, \phi )$ is zero: If $ \Delta ^{w} _{J} \phi \neq 0$, then $ J\subset I^0$, and so for any $ I\subset J$, 
we have $ \Delta ^{\sigma } _{I} I ^{0}=0$. The form $  B _{\Subset} ( I ^{0}, \phi )$ is controlled by the constant $ \mathbf B _{\Subset}$.  So our argument is complete.  

%%%%%%%%%%%%%%%%%%%%%%%%%%%%%% SECTION  SECTION SECTION
%%%%%%%%%%%%%%%%%%%%%%%%%%%%%% SECTION  SECTION SECTION 
\section{Dominating Functional Energy} %\label{s.}

We have two estimates of the functional energy constant to prove, that $ \mathbf F \lesssim \mathbf H+ \boldsymbol\Psi $ and 
$ \mathbf F \lesssim  \mathbf H +\mathbf B _{\Subset}$.  The data for these considerations are  
a non-negative function $ f$, its stopping intervals $ \mathcal F$,  and a sequence of functions $ \{g_F \;:\; F\in \mathcal F\}$, 
with  $ g_F$ $ \mathcal F$-adapted to $ F$.  
We assume that $ \mathbf H$ and  $ \mathbf B _{\Subset}$ are finite, and consider 
the expression $  B _{\Subset} (f,  g_F) $. By the monotonicity property Lemma~\ref{mono}, 
it suffices to assume that $ f$ takes the value $ \gamma (F)$ on each set of the form $ F - \bigcup \textup{Child} _{\mathcal F}(F)$. 
The point to observe is that if we write 
\begin{equation*}
\widetilde f _{F} = f \cdot (F - \bigcup\textup{Child}_{\mathcal F}(F)) + \sum_{F'\in \textup{Child} _{\mathcal F} (F)} \mathbb E ^{\sigma } _{F'} f \cdot F' \,, 
\end{equation*}
we can appeal to the bounded fluctuation property to write 
\begin{equation*}
\bigl\lvert B _{\Subset} (f \cdot F, g_F)\bigr\rvert \lesssim \mathbf {BF} 
\lVert \widetilde f _{F} \rVert_{\sigma } \lVert g_F\rVert_{w} 
\lesssim \{  \sqrt {\mathbf A_2} + \mathbf B _{\Subset}\} \lVert \widetilde f _{F} \rVert_{\sigma } \lVert g_F\rVert_{w}  \,. 
\end{equation*}
Applying Cauchy-Schwartz to the last two products of norms, we get 
\begin{align*}
\sum_{F\in \mathcal F}  \lVert \widetilde f _{F} \rVert_{\sigma } \lVert g_F\rVert_{w}  
& \le
\Biggl[ 
\sum_{F\in \mathcal F}  \lVert \widetilde f _{F} \rVert_{\sigma } ^2 
\Biggr] ^{1/2} \lVert g\rVert_{w}
\\
& \lesssim \bigl[ 
 \sum_{F\in \mathcal F}  \sigma (F) (\mathbb E ^{\sigma} _{F} \lvert  f\rvert ) ^2
 \bigr] ^{1/2} \lVert g\rVert_{w} \lesssim \lVert f\rVert_{\sigma } \lVert g\rVert_{w}
\end{align*}
where we write $ g= \sum_{F\in \mathcal F} g_F$.  Note that the condition that the functions $ g_F$ be $ \mathcal F$-adapted implies 
that they are orthogonal in $ F\in \mathcal F$.  

We conclude that on the assumption that $  \sqrt {\mathbf A_2} $ and $ \mathbf B _{\Subset}$ are finite, we have 
\begin{align*}
	\sum _{F\in \mathcal F} \bigl\lvert B _{\Subset} ( f - (f \cdot F), g_F)\bigr\rvert \lesssim \{\sqrt {\mathbf A_2} + \mathbf B _{\Subset}\} 
\lVert f\rVert_{\sigma } \lVert g\rVert_{w} \,. 
\end{align*}
With the assumptions on the functions $ g_F$, in the definition of functional energy, and the inequalities \eqref{e.Elower} 
are at our disposal.
This must be done in a way that controls the right-hand side of that inequality.  

Let  Haar multiplier $ \overline g_F$ of $ g_F$, chosen so that $ \langle x, \Delta _{J} ^{w} g_F \rangle _{w} \ge 0$ for all $ J$. 
We have from \eqref{e.Elower} the estimate 
\begin{gather} \label{e.oop}
 \sum_{J ^{\ast} \in \mathcal J ^{\ast} (F)} 
 \mathsf  P (( f - (f \cdot F)) \sigma , J ^{\ast} ) \bigl\langle \frac {x} {\lvert  J ^{\ast} \rvert },  g_F \bigr\rangle_{w}  
 \lesssim 
 B _{\Subset} ( f - (f \cdot F), \overline   g_F) + D (f,g_F)
 \\ \label{e.D}
D (f,g_F) := 
\sum_{t=1}^\infty 
\gamma (\pi ^{t} _{\mathcal F} F)  \sum_{\substack{J \;:\; J\Subset F\\ \pi _{F} J=F}} 
\Bigl[ \frac {\lvert  J\rvert} {\lvert  \pi ^{t} F\rvert } \Bigr] ^{1-\epsilon}P(\sigma\pi^{t} _{\mathcal F} F , J) \sqrt {w (J)} \lvert \langle g_F, h ^{w} _{J} \rangle _{w}\rvert 
\end{gather}
Note that we are appealing to the specific form of $ f$ to obtain the form for $ D (f, g_F)$.  
If the first term on the right in \eqref{e.oop} is the larger, we are finished with the proof. Otherwise, 
we will use the $ A_2$ constant to control the terms $ D (f,g_F)$.  

Note that by repeated application of  Cauchy-Schwartz in different variables, it holds that 
\begin{align}
	D (f,g_F) ^2   & \le \lVert g_F\rVert_{w} ^2 
	\Biggl\lvert 
\sum_{t=1}^\infty 
\gamma (\pi ^{t} _{\mathcal F} F)  \sum_{\substack{J \;:\; J\Subset F\\ \pi _{F} J=F}} 
\Bigl[ \frac {\lvert  J\rvert} {\lvert  \pi ^{t} F\rvert } \Bigr] ^{1-\epsilon}P(\sigma\pi^{t} _{\mathcal F} F , J) \sqrt {w (J)}
\Biggr\rvert ^2 
\\
& \lesssim 
\lVert g_F\rVert_{w} ^2 
	   \sum_{t=1}^\infty \gamma (\pi ^{t} _{\mathcal F} F) ^2  t ^2 
	\Biggl\lvert  \sum_{\substack{J \;:\; J\Subset F\\ \pi _{F} J=F}} 
	\Bigl[ \frac {\lvert  J\rvert} {\lvert  \pi ^{t} F\rvert } \Bigr] ^{1-\epsilon}P(\sigma\pi^{t} _{\mathcal F} F , J) \sqrt {w (J)}
\Biggr\rvert ^2 
\\
& \lesssim \lVert g_F\rVert_{w} ^2 
	   \sum_{t=1}^\infty \gamma (\pi ^{t} _{\mathcal F} F) ^2  t ^2  
	     \sum_{\substack{J \;:\; J\Subset F\\ \pi _{F} J=F}} 
	   \Bigl[ \frac  {\lvert  F\rvert }  { \lvert  J\rvert }  \Bigr] ^{1 + \epsilon }
	   \Bigl[ \frac {\lvert  J\rvert} {\lvert  \pi ^{t} F\rvert } \Bigr] ^{2-2\epsilon}P(\sigma\pi^{t} _{\mathcal F} F , J) ^2   {w (J)}
\\
& \lesssim  \mathbf A_ 2 \lVert g_F\rVert_{w} ^2 
\sum_{t=1}^\infty \gamma (\pi ^{t} _{\mathcal F} F) ^2  t ^2  2 ^{-t (1 - \epsilon )}
\sum_{\substack{J \;:\; J\Subset F\\ \pi _{F} J=F}} 
	   \Bigl[ \frac {\lvert  J\rvert} {\lvert  \pi ^{t} F\rvert } \Bigr] ^{1-3\epsilon}P(\sigma\pi^{t} _{\mathcal F} F , J) 
	   \\ \label{e.D<}
& \lesssim \mathbf A_2  \lVert g_F\rVert_{w} ^2 
	   \sum_{t=1}^\infty   2 ^{-t/4} \gamma (\pi ^{t} _{\mathcal F} F) ^2   \sigma (\pi ^{t} _{\mathcal F} F) \,. 
\end{align}
This holds if  $ 1- 3 \epsilon > \tfrac 13$, which we can assume is true, as this choice of $ \epsilon $ is only 
associated with the definition of being good. For any $ 0<\epsilon < 1 $, we can make a choice of integer $ r$ so that it 
suffices to consider only $ (\epsilon ,r)$-good intervals.  
From the (quasi)-orthogonality conditions on $ f$ and $ \{g_F\}$, it is then easy to see that 
\begin{equation*}
	\sum_{F\in \mathcal F} 	D (f,g_F) \lesssim \sqrt {\mathbf A_2} \lVert f\rVert_{\sigma } \Bigl\lVert \sum_{F\in \mathcal F} g_F  \Bigr\rVert_{w} 
\end{equation*}
This completes our proof of $ \mathbf F \lesssim \sqrt {\mathbf A_2} +\mathbf B _{\Subset}$.
Indeed, to verify the last inequality, let us write $ g= \sum_{F\in \mathcal F} g_F$, and apply \eqref{e.D<}. 
\begin{align*}
	\sum_{F\in \mathcal F} 	D (f,g_F) & \lesssim 
	\mathbf A_2 \sum_{F\in \mathcal F} \lVert g_F \rVert_{w}
	\Biggl[	   \sum_{t=1}^\infty   2 ^{-t/4} \gamma (\pi ^{t} _{\mathcal F} F) ^2   \sigma (\pi ^{t} _{\mathcal F} F)  \Biggr] ^{1/2} 
	\\
	& \lesssim 
	\sqrt {\mathbf A_2}  \lVert g\rVert_{w} 
		\Biggl[\sum_{F\in \mathcal F}\sum_{t=1}^\infty   2 ^{-t/4} \gamma (\pi ^{t} _{\mathcal F} F) ^2   \sigma (\pi ^{t} _{\mathcal F} F)  \Biggr] ^{1/2} 
\\
	& \lesssim 
\sqrt {\mathbf A_2} \lVert g\rVert_{w} 
	\Biggl[\sum_{F\in \mathcal F}  \gamma   F) ^2   \sigma (  F)  \Biggr] ^{1/2}  \lesssim
	\sqrt {\mathbf A_2} \lVert f\rVert_{\sigma } \lVert g\rVert_{w}\,. 
\end{align*}

\medskip 

It remains to prove $ \mathbf F \lesssim \boldsymbol \Psi $, which follows from an elementary application of the side condition.  
We reorganize the sum around the $ \mathcal F$-ancestors of an interval $ F \in \mathcal F$
\begin{align}
	\sum_{F\in \mathcal F} \sum_{J ^{\ast} \in \mathcal J ^{\ast} (F)}  &
	P ((f ({\mathbb R -F}) \sigma , J ^{\ast} ) \bigl\langle  \frac {x} {\lvert  J ^{\ast} \rvert }, g_F {J ^{\ast} }\bigr\rangle _{w}
	\\& \le 
	\sum_{F\in \mathcal F}
	\mathbb E _{F} ^{\sigma } f \cdot 
	\sum_{t=2}  ^{\infty } 
	\sum_{F' \in \mathcal F \;:\; \pi ^{t} _{\mathcal F} = F} 
	\sum_{J ^{\ast} \in \mathcal J ^{\ast} (F')}  	P (\sigma (F- F'), J ^{\ast} ) 
	\bigl\langle  \frac {x} {\lvert  J ^{\ast} \rvert }, g_{F'} {J ^{\ast} }\bigr\rangle _{w} 
	\\  \label{e.F2}
& \lesssim \boldsymbol \Psi 	\sum_{F\in \mathcal F}
\mathbb E _{F} ^{\sigma } f \cdot  \sigma (F) ^{1/2} 
	\sum_{t=2}  ^{\infty }  \psi (t) 
	\Biggl[	\sum_{F' \in \mathcal F \;:\; \pi ^{t} _{\mathcal F} = F} 
	 \sum_{J ^{\ast} \in \mathcal J ^{\ast} (F')}  
	 \lVert g_F \cdot J ^{\ast} \rVert_{w} ^2  
	\Biggr] ^{1/2} 
\end{align}
Here, we have appealed to \eqref{e.PsiI} with the data $ I_0 \leftarrow F$, $ \{I_j\} \leftarrow \{F' \in \mathcal F \;:\; \pi ^{t} _{\mathcal F} = F \}$, and $ \{I _{j,k}\} \leftarrow \bigcup \{ \mathcal J ^{\ast} (F') \;:\; F' \in \mathcal F \;:\; \pi ^{t} _{\mathcal F} = F\}$. 
It follows that we have $ \{I _{j,k}\}$ are $ (\epsilon ,r)$-good, and $ \lvert  I _{j,k}\rvert \le 2 ^{-t} \lvert  I _{j}\rvert  $, 
and therefore, we have 
\begin{equation*}
	\sum_{F' \in \mathcal F \;:\; \pi ^{t} _{\mathcal F} = F} 
	P (\sigma (F- F'), J ^{\ast} )  ^2 \sum_{J ^{\ast} \in \mathcal J ^{\ast} (F')}  \mathsf E (w, J ^{\ast} ) ^2  w ( J ^{\ast} )
	\le \boldsymbol \Psi ^2 \psi (t) \sigma (F) \,. 
\end{equation*}
The inequality \eqref{e.F2} then follows from Lemma~\ref{mono}, and the assumption that $ \mathbb E _{J ^{\ast} } ^{w} g _F=0$ for all $ F 
\in \mathcal F $, and  $ J ^{\ast} \in \mathcal J ^{\ast} (F) $.  In particular, we have 
\begin{equation*}
\bigl\lvert 
\bigl\langle  \frac {x} {\lvert  J ^{\ast} \rvert }, g_{F'} {J ^{\ast} }\bigr\rangle _{w}
\bigr\rvert \le \mathsf E (w,J ^{\ast} )  w (J ^{\ast} ) ^{1/2}   \lVert g_F \cdot J ^{\ast} \rVert_{w} \,. 
\end{equation*}

Now, we have the quasi-orthogonality 
estimate \eqref{e.F<}. 
Using the condition $ \sum_{t\ge 2} \psi  (t) =1$, we have 
\begin{align*}
	\sum_{F\in \mathcal F} &
\Biggl[ 
	\sum_{t=2}  ^{\infty }  \psi (t) 
	\Biggl[	\sum_{F' \in \mathcal F \;:\; \pi ^{t} _{\mathcal F} = F} 
	 \sum_{J ^{\ast} \in \mathcal J ^{\ast} (F')}  
	 \lVert g_F \cdot J ^{\ast} \rVert_{w} ^2 
\Biggr] ^{1/2}  \Biggr] ^2 
\\	& 
\le
	\sum_{F\in \mathcal F}
	\sum_{t=2}  ^{\infty }  \psi (t) 
	\sum_{F' \in \mathcal F \;:\; \pi ^{t} _{\mathcal F} = F} 
	 \sum_{J ^{\ast} \in \mathcal J ^{\ast} (F')}  
	 \lVert g_{F'} \cdot J ^{\ast} \rVert_{w} ^2 
	 \le \lVert g_F \rVert_{w} ^2 \,.   
\end{align*}

%%%%%%%%%%%%%%%%%%%%%%%%%%%%%% SECTION  SECTION SECTION
%%%%%%%%%%%%%%%%%%%%%%%%%%%%%% SECTION  SECTION SECTION 
\section{The Remaining Estimates} %\label{s.}

%%%%%%%%%%%%%%%%%%%%%%%%%%%%%% SUBSECTION SUBSECTION SUBSECTION SUBSECTION
%%%%%%%%%%%%%%%%%%%%%%%%%%%%%% SUBSECTION SUBSECTION SUBSECTION SUBSECTION 
\subsection{The Term $ B_{1,2}$.}\label{s.12}
The term $ B_{1,2}$ is defined in \eqref{e.12}; we are proving the estimate \eqref{e.12<}, where the constants on the 
right are the $ A_2$-constant and the weak-boundedness constant in \eqref{e.W}. 

Note that by the definition of the weak-boundedness constant,  we have 
\begin{align*}
\sum_{ (I,J) \;:\; \substack{ 3 I \cap 3 J \neq \emptyset \\  2 ^{-r} \lvert  I\rvert\le \lvert  J\rvert\le 2 ^{r} \lvert  I\rvert    }} 
\langle H _{\sigma} \Delta ^{\sigma}_I f , \Delta ^{w} _{J} \phi  \rangle _{w}  
& \le 
\mathbf W 
\sum_{ (I,J) \;:\; \substack{ 3 I \cap 3 J \neq \emptyset \\  2 ^{-r} \lvert  I\rvert\le \lvert  J\rvert\le 2 ^{r} \lvert  I\rvert    }} 
\bigl\lvert  \langle  f, h ^{\sigma } _{I} \rangle _{\sigma } 
\langle \phi , h ^{w } _{J} \rangle _{w }
\bigr\rvert 
\\
& \lesssim \mathbf W \lVert f \rVert_{\sigma } \lVert \phi \rVert_{w } \,. 
\end{align*}

If $ 3I \cap 3 J = \emptyset $, and $ \lvert  J\rvert \lesssim \lvert  I\rvert  $, we have the trivial consequence of \eqref{e.mono+P}, 
\begin{equation}\label{e.IJ}
\lvert \langle H _{\sigma } (I), h ^{w } _{J} \rangle _{w } \rvert
\lesssim \sigma (I) \frac { \sqrt {w (J)} \lvert  J\rvert } { (\lvert  J\rvert + \textup{dist} (I,J)) ^2 }
\end{equation}

Now, the estimate \eqref{e.IJ} implies that the remaining part of $ B_{1,2}$ is controlled by 
\begin{align*}
\sum_{ (I,J) \;:\; \substack{ 3 I \cap 3 J = \emptyset \\  2 ^{-r} \lvert  I\rvert\le \lvert  J\rvert\le 2 ^{r} \lvert  I\rvert    }} 
\langle H _{\sigma} \Delta ^{\sigma}_I f , \Delta ^{w} _{J} \phi  \rangle _{w}  
& \lesssim 
\sum_{ (I,J) \;:\; \substack{ 3 I \cap 3 J \neq \emptyset \\  2 ^{-r} \lvert  I\rvert\le \lvert  J\rvert\le 2 ^{r} \lvert  I\rvert    }} 
\bigl\lvert  \langle  f, h ^{\sigma } _{I} \rangle _{\sigma } 
\langle \phi , h ^{w } _{J} \rangle _{w } \bigr\rvert \cdot  \alpha (I,J)
\\
\alpha (I,J) & := 
\frac {\sigma (I) ^{1/2}  w (J) ^{1/2}   \lvert  J\rvert }    { (\lvert  J\rvert + \textup{dist} (I,J)) ^2 } \,. 
\end{align*}
The last coefficients satisfy the assumptions of Schur's test, with the relevant constant controlled by the $ A_2$ constant. Namely, for any $ I$ we have by the Cauchy-Schwartz inequality, 
\begin{align}
\Biggl[\sum_{ J \;:\; \substack{ 3 I \cap 3 J = \emptyset \\  2 ^{-r} \lvert  I\rvert\le \lvert  J\rvert\le 2 ^{r} \lvert  I\rvert    }} 
\alpha (I,J)  \Biggr] ^2 
& \lesssim   \label{e.schur}
\sum_{ (I,J) \;:\; \substack{ 3 I \cap 3 J = \emptyset \\  2 ^{-r} \lvert  I\rvert\le \lvert  J\rvert\le 2 ^{r} \lvert  I\rvert    }}  
\frac { \lvert  J\rvert } 
 { (\lvert  J\rvert + \textup{dist} (I,J)) ^2 }  
 \\& \qquad  \times 
\sum_{ (I,J) \;:\; \substack{ 3 I \cap 3 J = \emptyset \\  2 ^{-r} \lvert  I\rvert\le \lvert  J\rvert\le 2 ^{r} \lvert  I\rvert    }}  
  \frac {\sigma (I)   w (J)  }  { (\lvert  J\rvert + \textup{dist} (I,J)) ^2 }  
\lesssim \mathbf A_2 ^2  \,. 
\end{align}
This completes the proof of \eqref{e.12<}. 

%%%%%%%%%%%%%%%%%%%%%%%%%%%%%% SUBSECTION SUBSECTION SUBSECTION SUBSECTION
 %%%%%%%%%%%%%%%%%%%%%%%%%%%%%% SUBSECTION SUBSECTION SUBSECTION SUBSECTION 
\subsection{The Term $ B_{2,1}$}\label{s.21}

For the term $ B_{2,1}$ of \eqref{e.21}, we prove \eqref{e.21<}.  In this sum, we have $ 2 ^{r} \lvert  J\rvert \le \lvert  I\rvert  $, 
and $ J\cap 3I= \emptyset $.  We further restrict the length of $ J$ to be $ 2 ^{-s} \lvert  I\rvert $, for $ s\ge r$.  Using the 
estimate \eqref{e.IJ}, we can apply the Schur test to see that 
\begin{align*}
\sum_{ (I,J) \;:\; \substack{ 3 I \cap 3 J = \emptyset \\  2 ^{-r} \lvert  I\rvert\le \lvert  J\rvert\le 2 ^{r} \lvert  I\rvert    }} 
\langle H _{\sigma} \Delta ^{\sigma}_I f , \Delta ^{w} _{J} \phi  \rangle _{w}  
& \lesssim 2 ^{-s/2} \mathbf A_2 \lVert f\rVert_{\sigma } \lVert \phi \rVert_{w } \,. 
\end{align*}
Indeed, the only point to observe is that for the analog of the first term on the right in \eqref{e.schur}, that we have the geometric 
decay claimed above.  

%%%%%%%%%%%%%%%%%%%%%%%%%%%%%% SUBSECTION SUBSECTION SUBSECTION SUBSECTION
 %%%%%%%%%%%%%%%%%%%%%%%%%%%%%% SUBSECTION SUBSECTION SUBSECTION SUBSECTION 
\subsection{The Term $ B_{2,2}$}\label{s.22}
For the term $ B_{2,2}$ of \eqref{e.22}, we prove the second half of  \eqref{e.21<}. 
The distinction between this case and the previous is that this is the 'local' but not 'inside' part.  
For integers $ s\ge1$, we prove  
\begin{equation}\label{e.}
\sum_{I} \sum _{ J \;:\; \substack{ 2 ^{s} \lvert  J\rvert= \lvert  I\rvert  \\ J \subset 3I \backslash I  }} 
\langle H _{\sigma} \Delta ^{\sigma}_I f , \Delta ^{w} _{J} \phi  \rangle _{w}  
\lesssim 2 ^{-s/2}  \mathbf A_2 \lVert f\rVert_{\sigma } \lVert \phi \rVert_{w } \,. 
\end{equation}
But the essential points are on the one hand  that we have 
\begin{equation*}
\sum_{I} \sum _{ J \;:\; \substack{ 2 ^{s} \lvert  J\rvert= \lvert  I\rvert  \\ J \subset 3I \backslash I  }}  
\langle \phi ,h ^{w } _{J} \rangle _{w } ^2 \lesssim \lVert \phi \rVert_{w } ^2 \,, 
\end{equation*}
since the length and location  of $ J$ is specified by $ I$.  
And on the other hand, that we have the estimate 
\begin{equation}
\sum _{ J \;:\; \substack{ 2 ^{s} \lvert  J\rvert= \lvert  I\rvert  \\ J \subset 3I \backslash I  }}
\langle  H _{\sigma } h ^{\sigma } _{I}, h ^{w } _{J} \rangle _{w } ^2 
\lesssim \mathbf A_2 ^2 
2 ^{-s}  \,. 
\end{equation}

To be concrete, let $ \theta, \theta ' \in \{\pm\}$, and consider the sum 
\begin{equation*} 
\bigl\lvert \mathbb E _{I _{\theta }} ^{\sigma } h _{I} ^{\sigma }  \bigr\rvert ^2 
\sum _{ J \;:\; \substack{ 2 ^{s} \lvert  J\rvert= \lvert  I\rvert  \\ J \subset I + (\theta ' \lvert  I\rvert)   }}
\langle  H _{\sigma }  I _{\theta }, h ^{w } _{J} \rangle _{w } ^2 
\end{equation*}
Now, $\bigl\lvert \mathbb E _{I _{\theta }} ^{\sigma } h _{I} ^{\sigma }  \bigr\rvert \le \sigma (I _{\theta }) ^{-1/2} $, which is the estimate \eqref{e.h<}. 
And, we can apply \eqref{e.mono+P}  to see that 
\begin{align*}
S (I) = \sum _{ J \;:\; \substack{ 2 ^{s} \lvert  J\rvert= \lvert  I\rvert  \\ J \subset I + (\theta ' \lvert  I\rvert)   }}
\langle  H _{\sigma }  I _{\theta }, h ^{w } _{J} \rangle _{w } ^2 
& \lesssim 
\sum _{ J \;:\; \substack{ 2 ^{s} \lvert  J\rvert= \lvert  I\rvert  \\ J \subset I + (\theta ' \lvert  I\rvert)   }} 
\mathsf P (\sigma \cdot I _{\theta }, J ) ^2  w ( J ) 
\end{align*}
We ignore the contribution from energy.  But, the intervals $ J$ are good.  This means that $ \textup{dist} (J, I _{\theta }) 
\ge \lvert  I _{\theta }\rvert ^{1- \epsilon } \lvert  J\rvert ^{\epsilon } $, which fact we use to estimate the Poisson integral above as  
follows.  

%%%%%%%%%%%%%%%%%%%%%%%%%%%%%% LEMMA LEMMA LEMMA
\begin{lemma}\label{l.donotuse} Let $ J \subset I \subset I'$, $ \lvert  J\rvert = 2 ^{-s} \lvert  I\rvert  $,  with $ s\ge r$ and $ J$ good, then 
\begin{equation}  \label{e.donotuse}
\mathsf P (\sigma \cdot (I'-I) , J ) \le 2 ^{- (1- \epsilon ) s} \mathsf P (\sigma \cdot I ', I )\,.
\end{equation}
\end{lemma}
%%%%%%%%%%%%%%%%%%%%%%%%%%%%%% LEMMA LEMMA LEMMA

%%%%%%%%%%%%%%%%%%%%%%%%%%%%%% PROOF PROOF PROOF
\begin{proof}
Note that for $ x\in I'-I$ we have 
\begin{equation*}
\textup{dist} (x,J) \ge   \lvert  I\rvert ^{1- \epsilon } \lvert  J\rvert ^{\epsilon } 
= 2 ^{s (1- \epsilon )} \lvert  J\rvert\,.  
\end{equation*}
Using this in the definition of the Poisson integral, we get 
\begin{align*}
\mathsf P (\sigma \cdot (I'-I) , J ) 
& \le 2\int _{I'-I} \frac {\lvert  J\rvert } { \textup{dist} (x,J)^2 }  \; \sigma (dx) 
\\
& \lesssim \frac {\lvert  J\rvert } {\lvert  I\rvert } 
 \int _{I'-I} \frac {\lvert  I\rvert } { (\lvert  J\rvert  + \textup{dist} (x,J) ) ^2 } \; \sigma (dx) 
 \\
 & \lesssim 2 ^{- s (1- 2 \epsilon )} 
 \int _{I'-I} \frac {\lvert  I\rvert } { (\lvert  I\rvert  + \textup{dist} (x,I) ) ^2 } \; \sigma (dx) 
 = 2 ^{- s (1- 2 \epsilon )} \mathsf P (I , \sigma (I'-I)) \,. 
\end{align*}
\end{proof}
%%%%%%%%%%%%%%%%%%%%%%%%%%%%%% PROOF PROOF PROOF

Applying the Lemma, we have the estimate
\begin{equation*}
S (I) \lesssim  2 ^{-2 (1- 2\epsilon)s} 
 \mathsf P (\sigma \cdot I _{\theta }, I _{\theta} ) w (I + \theta ' \lvert  I\rvert ) 
 \lesssim \mathbf A _2 ^2 \sigma (I _{\theta }) \,. 
\end{equation*}

%%%%%%%%%%%%%%%%%%%%%%%%%%%%%% SUBSECTION SUBSECTION SUBSECTION SUBSECTION
%%%%%%%%%%%%%%%%%%%%%%%%%%%%%% SUBSECTION SUBSECTION SUBSECTION SUBSECTION 
\subsection{The Term $ B_{3,1}$}\label{s.31} 

For the term $ B_{3,1}$ of \eqref{e.31}, we prove \eqref{e.31<}.  This is a simple variant of the previous estimate. 
Namely,   we have for $ \theta \in \{\pm\}$, 
\begin{align*}
 \sum_{J\subset I ,\ I_J = I _{\theta }}
\langle  H _{\sigma }  I _{-\theta }, h ^{w } _{J} \rangle _{w } ^2 \lesssim \mathbf A_2 ^2 \sigma (I _{- \theta }) \,. 
\end{align*}
Here, by $ I _{- \theta }$ we mean the child of $ I$ complementary to $ I _{\theta }$.  We omit the details of the argument.

\end{document}